\documentclass[12pt,oneside,reqno]{amsart}
\usepackage[margin=1in]{geometry}
\usepackage[utf8]{inputenc}

\usepackage[backend=biber,style=numeric]{biblatex} 

  \usepackage[english]{babel}
\usepackage{csquotes}

\usepackage[font=small, labelfont=bf, labelsep=colon]{caption}

 \usepackage[font=small]{caption}
   \usepackage[scr=boondox
            ]   
           {mathalpha}

\usepackage[hidelinks]{hyperref}
 
\usepackage{amsmath, amsthm, amssymb,latexsym,epsfig,amsthm,enumerate,multicol,wasysym}
\usepackage{color}
\usepackage{graphicx}
\usepackage{nccrules}
\usepackage{xfrac}
\usepackage{hyperref}
\usepackage{textcomp}
\usepackage{tikz}
\usepackage{xcolor}
\usepackage{etoolbox}
\usepackage{comment}

\usepackage{soul}
\numberwithin{equation}{section}
\theoremstyle{plain}
\newtheorem{theorem}{Theorem}[section]
\newtheorem{lemma}[theorem]{Lemma}

\newtheorem{proposition}[theorem]{Proposition}

\theoremstyle{definition}
\newtheorem{definition}[theorem]{Definition}

\newtheorem{example}[theorem]{Example}

\theoremstyle{remark}

\newtheorem{case[theorem]}{Case}

\numberwithin{equation}{section}

\usepackage[normalem]{ulem}

\newcommand{\cI}{\mathcal I}

\author{Luis Gomez, Jonathan Jaimangal,  Azita Mayeli, Tasfia Proma}
\address{Department of Mathematics, CUNY Hunter College,  New York City, NY}
\email{luis.gomezreyes43@myhunter.cuny.edu}
\address{Department of Mathematics, CUNY Hunter College,  New York City, NY}
 \email{jonathan.jaimangal07@myhunter.cuny.edu}
\address{Department of Mathematics, CUNY Graduate Center, New York City, NY}
\email{amayeli@gc.cuny.edu } 
\address{Department of Mathematics, CUNY Hunter College,  New York City, NY}
\email{tasfia.ahmed67@myhunter.cuny.edu}

\thanks{A.~Mayeli was supported in part by AMS-Simons Research Enhancement Grant and the PSC-CUNY research grant 67807-00 55.}

\title{Eigenvalue distribution analysis of  multidimensional Prolate Matrices}

\begin{document}
  
\begin{abstract}
We extend classical time–frequency limiting analysis, historically applied to one dimensional finite signals, to the multidimensional (or multi-indexed) discrete setting. This extension is relevant for images, videos, and other multi‐dimensional signals, as it enables a rigorous study of joint time–frequency localization in higher dimensions.
To achieve this, we define multi‐dimensional time‐limiting and frequency‐limiting matrices tailored to signals on a Cartesian grid, then construct a multi-indexed prolate matrix. We prove that the spectrum of this matrix exhibits an eigenvalue concentration phenomenon: the bulk of eigenvalues cluster near 1 or 0, with a narrow transition band separating these regions. Moreover, we derive quantitative bounds on the width of the transition band in terms of time–bandwidth product  and prescribed accuracy.

Concretely, our contributions are twofold: (i) we extend Theorem 1.4 of \cite{israel2024eigenvalue} to the Cartesian discrete setting for higher dimensional signals; and (ii) within this framework,  we develop a multidimensional generalization of the non-asymptotic eigenvalue-distribution analysis for prolate matrices from \cite{karnik2019fast}.   The advances are summarized in Theorem \ref{main-theorem}. 
We test our theoretical results  through  numerical experiments in one‐ and two‐dimensional settings. The empirical results confirm the predicted eigenvalue concentration and illustrate potential applications in fast computation for image analysis, multi‐dimensional spectral estimation, and related signal‐processing tasks.  
 \end{abstract}
   \maketitle
\setcounter{tocdepth}{1}

\vspace*{-1em}  
\tableofcontents
\section{Introduction and main result}

\subsection{Motivation and Background}
\label{ssec:motivation}
In many signal processing and information-theoretic applications, one is often interested in signals that are simultaneously \emph{time-limited} and \emph{bandlimited}. However, the uncertainty principle imposes a fundamental limit: no nonzero function can be both perfectly time-limited and perfectly bandlimited. This trade-off leads to the study of signals that are \emph{optimally concentrated} in both domains.

In the one-dimensional discrete setting, consider signals   $x\in \mathbb C^N$ defined on a  discrete interval (or finite grid) \( [0,N]=\{0,\cdots, N-1\}\). Let \( \mathcal{I}_T \subset [0,N] \) denote a fixed time interval of length $M$ (e.g., \( \{0, \ldots, M-1\} \)), and let \( \mathcal{I}_B \subset [0,N]\) denote a frequency band of length $2K+1$ (e.g., low-frequency indices \( \{-K, \ldots, K\} \) modulo \( N \)). We say a signal  $x\in \mathbb C^N$ is bandlimited  with frequency band of size  $2K+1$ if its Fourier coefficients are zero for indices outside $ \mathcal{I}_B$. 

A natural question arises:

\begin{quote}
\emph{Among all signals bandlimited to \(\mathcal{I}_B\), which ones maximizes the concentration}
\[
\frac{\sum_{i \in \mathcal{I}_T} |x_i|^2}{\sum_{i \in [0,N]} |x_i|^2}\,?
\]

\end{quote}

This question is known as {\it the concentration problem},  and it  seeks bandlimited signals that are maximally concentrated within a fixed time interval.  In a seminal series of papers, Slepian, Landau, Pollak, and Widom \cite{Bell1,Bell2,Bell3,LandauWidom80}  studied the continuous version of this problem extensively, establishing the foundational theory of prolate spheroidal wave functions (PSWFs) as the solutions to this time-frequency concentration problem. These functions possess remarkable spectral concentration properties and have the double orthogonality property, i.e., they form an orthogonal basis on both the real line and finite intervals. Later contributions by Daubechies and Widom extended this theory, offering further insights into the spectral behavior and asymptotics of the associated integral operators \cite{daubechies88,widom2006asymptotic}. In the discrete-time setting, Slepian also addressed the analogous problem, deriving discrete prolate spheroidal sequences. These sequences are the eigenfunctions of a compact, Hermitian operator, and they form an orthonormal basis of optimally concentrated functions in both space and frequency \cite{slepian1978-V}.

Collectively, this body of work forms a framework in the theory of time-frequency localization, with wide-ranging applications, such as  applications in scientific imaging problems, e.g., in cryoelectron microscopy (cryo-EM) and MRI,  signal compression in MRI  \cite{Yang02}, clustering and principal component analysis for cryo-EM imaging data \cite{Shkolnisky17A, Shkolnisky17B} and also for the analysis of the alignment problem in cryo-EM \cite{lederman2017numerical}; see also \cite{LedermanSinger17, LedermanSinger20}.

\medskip

Let  $[0,N]=\{0, 1, \cdots, N-1\}$. Let    \( \mathcal{I}_T, \mathcal{I}_B \subset [0,N]\), defined as above,      denote the time and frequency support sets, and let \( T_{\mathcal{I}_T} \) and \( B_{\mathcal{I}_B} \) denote the corresponding time- and frequency-limiting operators (see Section \ref{limiting-operators} for the definitions). These operators are projections onto the spaces of vectors with time support in $\cI_T$  and frequency support in $\cI_B$, respectively. 
   Using these projection operators, the concentration problem can be written as the following optimization:
\begin{align}\label{optimization-problem}
\max_{x\in\mathbb{C}^{N}}\;&\|T_{\mathcal{I}_T}x\|^{2}\\[4pt]
\text{subject to}\;& B_{\mathcal{I}_B}x = x, \qquad \|x\|^{2}=1.
\end{align}

This optimization problem 
  reduces to the study of spectral analysis of the Hermitian matrix  
\[
A = T_{\mathcal{I}_T} B_{\mathcal{I}_B} T_{\mathcal{I}_T}.
\]
 Indeed,  
let $x\in \mathbb C^N$ be a unit vector,  $\|x\|=1$,  such that $B_{\cI_B}(x)= x$. Then $$\|T_{\mathcal{I}_T}x\|_{\ell^2(\mathbb C^N)}^{2} = \langle  T_{\mathcal{I}_T} B_{\mathcal{I}_B} T_{\mathcal{I}_T} x, x\rangle.$$


The composition $T_{\mathcal{I}_T} B_{\mathcal{I}_B} T_{\mathcal{I}_T}$ defines  a positive semi-definite self-adjoint operator. 
 The eigenvectors of \( A \) are known as the discrete prolate spheroidal sequences (DPSS), and they are bandlimited and most concentrated within the time interval $\cI_T$. Therefore, 
 they constitute the solutions to the optimization problem~\eqref{optimization-problem}.
On the other hand,  
the eigenvalues quantify the degree of time concentration.

\vskip.12in



The discrete version of the concentration problem and the eigenvalue distribution of time-frequency limiting operators have been considered in one dimension \cite{Bell5, Karnik21}. 
In many modern applications, however, the data are inherently multidimensional.  See e.g., 
\cite{Yang02, Shkolnisky17A, Shkolnisky17B,  lederman2017numerical,LedermanSinger17, LedermanSinger20} and references therein.
For example, in image processing,
 images are naturally two-dimensional, and their effective analysis, compression, and restoration often rely on understanding how energy is distributed across both spatial and frequency domains.
   Applications such as video processing, volumetric data in medical imaging, and sensor array measurements involve signals defined on higher dimensional grids, where classical one-dimensional methods fall short.  Efficient compression schemes benefit from exploiting the inherent redundancy in multidimensional (or multi-indexed) data, which can be better characterized by extending time–frequency localization results to higher dimensions.
 Extending the one-dimensional theory to a multidimensional framework is thus both a natural progression and a necessity for addressing these contemporary challenges.

\medskip  
In this paper, we extend the framework of the one-dimensional time-frequency concentration problem from a finite setting to a multidimensional  (or a higher dimensional) finite setting. We then investigate the eigenvalue distribution, the clustering behavior of eigenvalues near $0$
and $1$, and characterize the associated transition band. These results generalize those of \cite{Karnik21} to higher‐dimensional discrete settings, and can be viewed as the discretized analogues of the higher dimensional Euclidean  Theorem 1.4 proved in \cite{israel2024eigenvalue} .
 
\medskip

\subsection{Main Result}
\label{ssec:main-results} 
Fix integers  $d\geq 1$ and $N, M, K$ satisfying $$M< N, \quad    K\le \Bigl\lfloor\frac{N-1}{2}\Bigr\rfloor.$$ 

Set the frequency width 
  $$W:=\frac{2K+1}{2N}\in (0,1/2)$$
 
Write the $d$‑dimensional spatial and frequency ``cubes"  $\pmod{N}$ as
\[
\mathbf M= [0,M]^d := \{0,\dots,M-1\}^d,
\qquad
\mathbf K= [-K,K]^d := \{-K,\dots,K\}^d,
\]

Let \(T_{\mathbf{M}}\) and \(B_{\mathbf{K}} = B_{\mathbf{K},W}\) denote the limiting operators defined in Section~\ref{limiting-operators}.
 Let  $A = T_{\mathbf M}\,B_{\mathbf K}\,T_{\mathbf M}$, be  the prolate matrix associated  with the cubes  \([0,M]^d\) and \([-K,K]^d\). 
The matrix \(A\) is positive definite, its eigenvalues (which depend on \(\mathbf M\) and \(\mathbf K\)) lie in \([0,1]\) (see Theorem \ref{A-is-selfadjoint} below), and we denote the positive ones in non-increasing order by

 \[
1 \;>\;\lambda_{\mathbf N}^{(1)}(\mathbf M, \mathbf K)\;\ge\;\cdots\;\ge\;\lambda_{\mathbf N}^{(\ell)}(\mathbf M, \mathbf K)\;>\;0.
\]
To simplify notation, we drop the dependence of the eigenvalues on \(\mathbf M\) and \(\mathbf K\); henceforth, we write
\[
  \lambda_{\mathbf N}^{(k)}=\lambda_{\mathbf N}^{(k)}(\mathbf M,\mathbf K).
\]
For \(\epsilon>0\) set
\begin{align*}
&\mathscr{m}_\epsilon(\textbf{M,K}) := \sharp \{r\in \mathbb N: ~ \lambda_{\bf N}^{(r)} >  \epsilon\},  \quad  \epsilon\in (0,1)\\
&\mathscr{n}_\epsilon(\textbf{M,K}) := \sharp \{r\in \mathbb N: ~ \lambda_{\bf N}^{(r)}  \in (\epsilon,1-\epsilon) \}, \quad \epsilon \in (0, 1/2).
\end{align*}
$ 
\mathscr{m}_\epsilon(\textbf{M,K}) $  is the number of the eigenvalues  $\epsilon$-close to 1, and $\mathscr{n}_\epsilon(\textbf{M,K})$ is the  number of the eigenvalues 
  in the `transition band' (a.k.a `plung region'). 

 \medskip 
 
 Our main result follows.

 \begin{theorem}[Spectrum Concentration]\label{main-theorem}

There exists a constant \(C_d>0\) depending only on \(d\) such that for every
 $\epsilon>0$ and $K\le \Bigl\lfloor\frac{N-1}{2}\Bigr\rfloor$
  \begin{align}\label{M-2MW1}
&\Bigl|  \mathscr{m}_\epsilon({\bf M}, {\bf K}) - (2MW)^d\Bigr| \leq C_d 
B_d(MW, \epsilon),    & \epsilon\in(0,1) \\\label{N} 
& \mathscr{n}_\epsilon(\mathbf M,\mathbf K)
\;\leq \;C_d \; B_d(MW, \epsilon),   & \epsilon\in(0,\tfrac12)
\end{align}
where 

 $$B_d(MW, \epsilon):= \log(MW)\,\log\!\bigl(\tfrac1\epsilon\bigr)\,
\max\Bigl\{\bigl[\log(MW)\,\log\tfrac1\epsilon\bigr]^{\,d-1},\;(2MW)^{\,d-1}\Bigr\}. 
$$ 
\end{theorem}

The bound $B_d(MW, \epsilon)$ scales optimally in dimension \(d\) as
\[
  O\!\bigl((\log(MW)\,\log\tfrac1\varepsilon)^{d}\bigr),
\]
while still explicitly reflecting the role of the time–bandwidth product \(MW\).

The result state that up to an explicitly controlled error \(B_d\), exactly \((2MW)^d\)
eigenvalues are \(\epsilon\)-close to~1, at most \(B_d\) lie in the
\emph{transition band} \((\epsilon,1-\epsilon)\), and the remainder are
\(\le\epsilon\).
 We illustrate this phenomenon in dimension $d=1$ in Figure \ref{figure:prolate9}.

\medskip

 To prove Theorem \ref{main-theorem}, observe that the multi‐indexed matrix $A$ (acting on $\mathbb{C}^{M_1\times \cdots \times M_d}$) can be expressed as the $d$‐fold composition of one‐dimensional time‐ and frequency‐limiting matrices, each acting on a single variable. Consequently, estimates for the eigenvalues $\lambda_{\mathbf{N}}^{(r)}$ follow directly from the corresponding one‐dimensional results.
Here, we borrow freely from the techniques used in the third listed author's earlier work \cite{israel2024eigenvalue} and extend Theorem 1.4 of that paper to the case of multi-dimensional discrete signals.
 In particular, we apply the one-dimensional transition‐region bounds from \cite{Karnik21} (recalled in Section \ref{prolatematrix}), together with Proposition \ref{1d:prop} (which we prove in Section \ref{auxilliary-results}) to handle those eigenvalues near 1 and those in the transition region in the multidimensional settings. The full details appear in Sections  \ref{auxilliary-results} and \ref{proof-of-main-theorem}.

\subsection{Related work on 1D prolate matrices}

The concentration operators associated with discrete prolate spheroidal sequences (DPSSs) in the discrete‐time setting have been studied extensively since the foundational work of Landau, Pollak, and Slepian in the 1960s and 1970s.  In particular, the eigenvalue distributions of the corresponding prolate matrices exhibit a characteristic clustering behavior in which most eigenvalues lie very close to either 0 or 1, with only a small “transition band’’ between these extremes.  We summarize here the key asymptotic and non‐asymptotic results on the eigenvalue distributions of both DPSS and PSWF operators.

\subsection{DPSS Eigenvalues}
 
Let \(M\in\mathbb{N}\) and \(W\in(0,\tfrac12)\).  The \(M\times M\) prolate matrix 
\[
B_{M,W}[m,n] \;=\; \frac{\sin\bigl(2\pi W(m-n)\bigr)}{\pi\,(m-n)}, \qquad m,n=0,\dots,M-1,
\]
arises as the matrix representation of the operator \(T_M B_W T_M\), where \(T_M:\ell^2(\mathbb Z)\to \ell^2(\mathbb Z)\) is the time‐limiting operator (truncation to indices \(0,\dots,M-1\)) and \(B_W:\ell^2(\mathbb Z)\to \ell^2(\mathbb Z)\) is the band‐limiting operator (restriction of the discrete‐time Fourier transform $\mathcal F: \ell^2(\mathbb Z) \to L^2(0,1)$ to \(\lvert \xi\rvert\le W\)).  Its eigenvalues satisfy   
\[
1 \;>\; \lambda_0(M,W)\;>\;\lambda_1(M,W)\;>\;\cdots>\;\lambda_{M-1}(M,W)\;>\;0
\]
and lie strictly between $0$  and $1$,  and satisfy \(\sum_{k=0}^{M-1}\lambda_k(M,W)=2MW\).  Empirically and asymptotically, one observes that approximately \(2MW\) eigenvalues lie very close to $1$, approximately \(M-2MW\) lie very close to $0$, and only \(O(\log(MW)\,\log(1/\varepsilon))\) lie in any fixed ``transition interval" \((\varepsilon,1-\varepsilon)\)  (\cite{slepian1978-V,Karnik21}). 

\begin{itemize}
  \item \textbf{Asymptotic clustering.}  Slepian  \cite{slepian1978-V} showed that for any fixed \(W\in(0,\tfrac12)\) and any real \(b\),
  \begin{align}\label{sigmoid-estimate}
  \lambda_{\lfloor\,2MW + \tfrac{b}{\pi}\,\log M\rfloor}(M,W)\;\sim\;\frac{1}{1+e^{\,b\pi}}
    \quad\text{as }M\to\infty,
\end{align}
  which implies for any fixed \(\varepsilon\in(0,\tfrac12)\),
  \[
    \#\bigl\{\,k:\,\varepsilon<\lambda_k(M,W)<1-\varepsilon\,\bigr\}
    \;\sim\;\frac{2}{\pi^2}\,\log M\;\log\!\Bigl(\tfrac{1}{\varepsilon}-1\Bigr)
    \quad\text{as }M\to\infty.
  \]
  Hence, the width of the “transition band’’ grows on the order of \(\log M\log(1/\varepsilon)\) but does not explicitly capture the dependence on \(W\).  Moreover, Slepian gave the approximation
  \[
    \lambda_{k}(M,W)\;\approx\;\Bigl[\,1 + \exp\!\Bigl(
    \tfrac{
    -\pi^2\bigl(2MW - k - \tfrac12\bigr)}{\log(8M\sin(2\pi W)) + \gamma}\Bigr)\Bigr]^{-1},
  \]
  valid for \(\lambda_k\in(0.2,0.8)\), where \(\gamma\approx0.5772\) is the Euler–Mascheroni constant.  This formula suggests 
  \(\#\{\,k:\,\varepsilon<\lambda_k(M,W)<1-\varepsilon\}\approx \tfrac{2}{\pi^2}\,\log\bigl(8e^\gamma M\sin(2\pi W)\bigr)\,\log\!\bigl(\tfrac{1}{\varepsilon}-1\bigr)\)
  for \(\varepsilon\in(0.2,0.5)\), correctly capturing the logarithmic dependence on \(M\), \(W\), and \(\varepsilon\).  

  \item \textbf{Early non‐asymptotic bounds.}  
  Zhu and Wakin \cite{ZhuWakin2017}  proved that for all integers \(M\ge2,\;W\in(0,\tfrac12)\), and \(\varepsilon\in(0,\tfrac12)\),
  \[
    \#\bigl\{\,k:\,\varepsilon\le\lambda_k(M,W)\le 1-\varepsilon\bigr\}\;\le\; \frac{2}{\pi^2}\, 
    \left(
     \log(M-1)\;+\;\frac{\,2M-1\,}{M-1}\right)\;\frac{1}{\varepsilon\,(1-\varepsilon)}\,
  \]
  which highlights the \(\log M\) dependence but fails to reflect any dependence on \(W\).  In particular, when \(\varepsilon\) is small, the \(O(1/\varepsilon)\) term dominates, and the bound can exceed \(M\).  

  \item \textbf{  Non‐asymptotic bounds-explicit $W$-dependence.}  
  Boulsane, Bourguiba, and Karoui \cite{boulsane2020discrete} refined this by showing
  \[
\begin{aligned}
\#\bigl\{\,k:\,\varepsilon \le \lambda_k(M,W) \le 1-\varepsilon \bigr\}
  &\;\le\;
     \Bigl(
        \tfrac{1}{\pi^2}\,\log(2MW) + 0.45 - \tfrac23 W^2
        + \tfrac{\sin^2(2\pi MW)}{6\pi^2 M^2}
     \Bigr)
     \frac{1}{\varepsilon\,(1-\varepsilon)}
\end{aligned}
\]

  valid for all \(M\ge1,\;W\in(0,\tfrac12),\;\varepsilon\in(0,\tfrac12)\).  As \(M\to\infty\) with fixed \(W\), the dominant term is \(\tfrac{1}{\pi^2}\log(2MW)\), but unfortunately the dependence on \(\varepsilon\) remains \(O(1/\varepsilon)\).  \medskip 

  \item \textbf{Fast Slepian transform bounds.}  
  In \cite{karnik2019fast}, Karnik \emph{et. al.} proved that for all \(M\in\mathbb {N},\;W\in(0,\tfrac12),\;\varepsilon\in(0,\tfrac12)\),
  \[
    \#\bigl\{\,k:\,\varepsilon<\lambda_k(M,W)<1-\varepsilon\bigr\}
    \;\le\;\Bigl(\tfrac{8}{\pi^2}\,\log(8M) + 12\Bigr)\;\log\!\Bigl(\tfrac{15}{\varepsilon}\Bigr). 
  \]
  This bound captures the correct \(\log M\) dependence on the matrix size and introduces the optimal \(O(\log(1/\varepsilon))\) dependence on \(\varepsilon\), but with a leading constant \(8/\pi^2\) that is four times larger than the asymptotic constant \(2/\pi^2\).  The bound also fails to reflect the dependence on $W$.  

\item \textbf{Improved non-Asymptotic transition bounds.}
 In \cite{Karnik21}, Karnik \emph{et al.} show that for all integers \(N\), bandwidths \(W\in(0,\tfrac12)\), and thresholds \(\varepsilon\in(0,\tfrac12)\),
   \[
      \#\bigl\{\,k : \varepsilon < \lambda_k(N,W) < 1 - \varepsilon\bigr\}
      \;\le\;
      \frac{2}{\pi^2}\,\log\bigl(100\,M\,W + 25\bigr)\,\log\!\Bigl(\frac{5}{\varepsilon(1-\varepsilon)}\Bigr)
      \;+\;7.
    \]
    This bound achieves the optimal \(O(\log MW\,\log\tfrac1\varepsilon)\) scaling  and  captures the role of the bandwidth \(W\).  
\end{itemize} 
\medskip

 Despite these advances, obtaining non-asymptotic bounds for the transition band in higher dimensions has, to the best of our knowledge, remained an open problem. In this special case, we close that gap by establishing bounds of the form 
$$ 
\sharp \{k\in \mathbb N: ~ \lambda_{\bf N}^{(k)}  \in (\epsilon,1-\epsilon) \} \leq C_d 
B_d(MW,\epsilon)$$
as we prove in Theorem 
\ref{main-theorem} 
\eqref{N}.   
 This bound scales optimally in dimension \(d\) as
\(
  O\!\bigl((\log(MW)\,\log\tfrac1\varepsilon)^{d}\bigr) 
\) scaling  and  captures the role of the time-bandwidth \(MW\).  

In the same theorem,  we also give a non‐asymptotic estimate for the number of eigenvalues that are 
 $\epsilon$-close to 1.

For the continuous‐time concentration operators associated with prolate spheroidal wave functions (PSWFs) in both one dimension and higher dimensions, see e.g.,  \cite{israel2024eigenvalue,hughes2024eigenvalue,israel2015eigenvalue,marceca2024eigenvalue} and references therein.

 \subsection{Organization of the paper}
\label{ssec:organization}

 The remainder of this paper is organized as follows: 
 \begin{itemize}
     \item 
 In Section \ref{background}, we fix multi‐index conventions, recall the one‐dimensional DFT and Dirichlet kernel facts, and introduce the 1D time‐ and band (or frequency)‐limiting projections. Lemmas \ref{timelimitopertorisaprojection}-\ref{projection} collect the key spectral properties in 1D that we will build upon to establish our results in higher dimensions.  
 \item In Section \ref{prolatematrix},  we define the 
$d$ dimensional 
 time‐ and frequency‐limiting operators, and assemble the generalized prolate operator. We prove that this operator is Hermitian-Toeplitz and show that its eigenvalues factor into products of 1D eigenvalues (Lemma \ref{spectrum-of-tensors}).  Basic eigenvalue behavior and multiplicity counts in higher dimensions are also established. 

\item In Section \ref{auxilliary-results}, 
 we derive non‐asymptotic bounds on the number of eigenvalues in the ``transition band", building 
on Sections \ref{background}–\ref{prolatematrix}.  Lemmas \ref{lem1}–\ref{lem2} handle the upper‐ and lower‐tail estimates. These combine to yield the precise eigenvalue‐clustering bound stated in Theorem \ref{main-theorem}.
\item In Section \ref{proof-of-main-theorem}
 we complete the proof of the main Theorem \ref{main-theorem}.
 \item 
 Appendix ~ \ref{appendix:proof-isomorphism}
  contains proof of some  results    referenced in Sections \ref{background}-\ref{auxilliary-results}, and  Appendix~\ref{appendix:figures} contains all numerical illustrations of the results and observations.
  \end{itemize}

 \section*{Acknowledgements}      
The authors would like to thank  Arie Israel and Kevin Hughes for their valuable input and discussions during the preparation of this paper.

\section{Notations and preliminaries}\label{background}

 \subsection{Multi-indexed signals}
 Let \( d \in \mathbb{N} \) denote the number of dimensions. 
In this paper, we work in a \( d \)-dimensional setting and adopt the following notation:

    A multi-index is denoted by a bold letter. For example, \(\mathbf{n} = (n_1, n_2, \ldots, n_d)\) represents a \(d\)-dimensional index where each \( n_i \) is a    nonnegative integer. We denote the set of all such indices by \(\mathbb{Z}^d\) (or \(\mathbb{N}^d\) when nonnegative indices are needed).
\vskip.1in 

For an integer $N$, we define the discrete interval $[0,N]=\{0, 1, \cdots, N-1\}$.  
For $d\geq 1$ and  the integers $N_1, \cdots, N_d$, we define the  $d$-dimensional  {\it discrete cube}  (or grid),  as 
\begin{align}\label{grid}
\mathcal{I} =  [0,N_1]\times \cdots \times [0,N_d]. 
 \end{align}
 
 \vskip.1in

    A signal  indexed by   the cube $\cI$ 
      is   represented as \[ x = \{ x[{\mathbf{n}}] \}_{\mathbf{n} \in \mathcal{I}}\in \mathbb C^{N_1\times \cdots \times N_d}. \] 
      
      For the index set \(\mathcal I\), we denote the space of all such signals by
\(\mathbb C^{N_1\times\cdots\times N_d}\)
(or \(\mathbb R^{N_1\times\cdots\times N_d}\) in the real case).  

      When $d=1$ and $N_1=N$, the signal $x\in \mathbb C^N$ is a complex vector of finite length $N$, i.e., $x=(x(0), x(1), \cdots, x(N-1))$.   When $d=2$, the signal $x$ can be  naturally identified with an 
 $N_1\times N_2$ matrix, where the entries of $i$th row and $j$th column is $x(i,j)$. It is custom to denote the space of complex (or real) matrices  by $\mathbb C^{N_1\times N_2}$ (or $\mathbb R^{N_1\times N_2}$).

 \vskip.1in

    For vectors \( x \in \mathbb{C}^{N_1 \times \cdots \times N_d} \), the \( \ell_2 \)-norm is defined as
    \[
    \|x\|_2 = \left( \sum_{\mathbf{n} \in \mathcal{I}} \left|x[\mathbf{n}]\right|^2 \right)^{1/2}.
    \]
    

\vskip.1in 
When $d=1$, any linear map $ \mathbb C^N\to \mathbb C^N$ can be considered as a $N\times N$ matrix. For $d=2$, and higher, any linear map  $ \mathbb C^{N_1\times \cdots \times N_d} \to \mathbb C^{N_1\times \cdots \times N_d}$ 
 can be viewed as a multidimensional array with $2d$ indices.

 \medskip


 \subsection{Tensor products}
Let \( \mathbf{u}, \mathbf{v} \in \mathbb{C}^d \). The tensor product (also called the outer product) of \( \mathbf{u} \) and \( \mathbf{v} \), denoted \( \mathbf{u} \otimes \mathbf{v} \), is a $d\times d$ matrix with entries in $\mathbb  C$ defined by:
\begin{align}\label{tensor-prod-2vecs}
(\mathbf{u} \otimes \mathbf{v})_{i,j} = u_i \cdot v_j, \quad \text{for } i, j = 1, \dots, d
\end{align}

This definition can be naturally generalized to more than two vectors. Given vectors \( \mathbf{u}^{(1)}, \mathbf{u}^{(2)}, \dots, \mathbf{u}^{(k)} \in \mathbb{C}^d \), their tensor product is a \( k \)-th order tensor of shape \( d \times d \times \cdots \times d \) (with \( k \) modes), defined by:
\[
\left( \otimes_{i=1}^k \mathbf{u}^{(i)} \right)_{n_1, \dots, n_k} = \prod_{i=1}^k u^{(i)}_{n_i}, \quad \text{for } n_i = 1, \dots, d
\]

Let $V_1,\dots,V_d$ be vector spaces over a common field $\mathbb{F}$ ($\mathbb R$ or $\mathbb C$), each equipped with an inner product $\langle\cdot,\cdot\rangle_i$.
For $u_i,v_i\in V_i$ ($1\le i\le d$), define the inner product between the elementary tensors
\[
  u \;=\; u_1 \otimes \cdots \otimes u_d,
  \qquad
  v \;=\; v_1 \otimes \cdots \otimes v_d\]
by
\!
\[ \langle u, v \rangle \;:=\; \prod_{i=1}^{d} \langle u_i, v_i \rangle_i.
\]
\!
This rule extends by bilinearity to the full tensor product $\otimes_{i=1}^{d} V_i$.

 \medskip

      Let \( T, \, S: \mathbb C^d \to \mathbb C^d \)   be linear maps between vector spaces over a common field \( \mathbb{F} \) ($\mathbb R$ or $\mathbb C$). The \emph{tensor product} of the linear maps \( T \) and \( S \), denoted  by \( T \otimes S \), is the linear map
\begin{align*} 
T \otimes S : \mathbb C^d  \otimes \mathbb C^d   \longrightarrow \mathbb C^d  \otimes \mathbb C^d 
 \end{align*}
defined on pure tensors $v\otimes v'$, $v, v'\in \mathbb C^d$,  by
\begin{align}\label{tensor-product-operators}
(T \otimes S)(v \otimes v') = T(v) \otimes S(v'),
\end{align}
and extended linearly to all of \(\mathbb C^d\otimes \mathbb C^d\).    

 For any maps $T', S': \mathbb C^d\to \mathbb C^d$,  the composition of $T'\otimes S'$ and $T\otimes S$ is given by 
\!
\begin{align}\label{composition-of-tensored-op}
(T\otimes S) \circ (T'\otimes S') = (T\circ T') \otimes (S\circ S').
\end{align}
 
These definitions naturally generalize to more than two vector spaces and maps.


\medskip

 \subsection{Multi-indexed discrete Fourier transform (DFT):}
Consider a $1$-dimensional signal $x\in \mathbb C^N$. 
 The   (1D)  DFT of \( x\)  is denoted by $F_N(x)=\widehat{x}$ and defined as
$$F_N(x)(k)=\widehat{x}(k) = \cfrac{1}{\sqrt{N}} \sum_{n=0}^{N-1}   x(n) e^{-j2\pi \frac{kn}{N}}, \qquad k\in [0,N].$$
This definition can be extended into the signals with an index in a higher dimensional discrete cube: 
Consider a \(d\)-dimensional signal  \( x=\{x[\mathbf{n}]\} \) defined on the grid  (or discrete cube) $\cI$ defined in \eqref{grid}, 
where  \(\mathbf{n} = (n_1, n_2, \ldots, n_d)\in \cI\). The  multidimensional DFT of \( x\)  is denoted by ${\bf F_N}(x)= \widehat{x}$ and defined as
\begin{align}\label{DFT}
{\bf F_N}(x)({\bf k})=\widehat{x}[\mathbf{k}] = \frac{1}{\sqrt{N_1 \cdots N_d}} \sum_{n_1=0}^{N_1-1} \cdots \sum_{n_d=0}^{N_d-1} x[n_1, \ldots, n_d] \, e^{-j 2\pi \left(\frac{k_1 n_1}{N_1} + \cdots + \frac{k_d n_d}{N_d}\right)},
 \end{align}
for \(\mathbf{k} = (k_1, k_2, \ldots, k_d)\) with $k_i$ in the discrete interval 
\([0,N_i]\). The multidimensional DFT can be computed as a sequence of one-dimensional DFTs along each dimension. This separability property is fundamental for efficient computation using algorithms such as the Fast Fourier Transform (FFT). 
    The original signal  
    $x=\{x[{\bf n}]\}_{\bf n\in \cI}$ 
 can be recovered at each  `position' ${\bf n}$ via the inverse DFT (IDFT) as follows:
    \[
    x[\mathbf{n}] = \frac{1}{\sqrt{N_1 \cdots N_d}} \sum_{k_1=0}^{N_1-1} \cdots \sum_{k_d=0}^{N_d-1} \hat{x}[k_1, \ldots, k_d] \, e^{j 2\pi \left(\frac{k_1 n_1}{N_1} + \cdots + \frac{k_d n_d}{N_d}\right)}.
    \]

    The DFT is a unitary transformation. In particular, it preserves the energy of the signal, which is  known as Plancherel identity: 
    \[
    \sum_{\mathbf{n} \in \mathcal{I}} |x[\mathbf{n}]|^2 =  
\sum_{\mathbf{k} \in \mathcal{I}} |\hat{x}[\mathbf{k}]|^2.
    \]

   Notice,  both \( x[\mathbf{n}] \) and \( \hat{x}[\mathbf{k}] \) are periodic in each dimension with periods \( N_1, N_2, \ldots, N_d \), respectively.  

   When $d=1$, then the (1D) DFT is given by a square and unitary matrix. For the simplicity, let $N_1=N$.  If we define this matrix by $F_{N}$, then for any $n, m\in [0,N]$,    the $n$-th row and $m$-the column is given by $[F_{N}]{(n,m)}= \frac{1}{\sqrt{N}} e^{-j2\pi \frac{nm}{N}}$, and its IDFT is given by 
$ [F_{N}^{-1}](n,m)= \frac{1}{\sqrt{N}} e^{j2\pi \frac{nm}{N}}$.

\vskip.1in

By   
  \eqref{tensor-product-operators},
the higher (multi)-dimensional discrete Fourier transform    \eqref{DFT}    in dimension \( d \) {\it can be expressed} as the \( d \)-fold tensor (or Kronecker) product of one-dimensional DFTs. Specifically, let \( \mathbf{F}_{\mathbf{N}} \) denote the (dD) DFT acting on the signal space \( \mathbb{C}^{N_1 \times \cdots \times N_d} \). Identifying this vector space isomorphically with the tensor product
\[
\mathbb{C}^{N_1} \otimes \cdots \otimes \mathbb{C}^{N_d},
\]
it follows that the multidimensional DFT is given by
\[
\mathbf{F}_{\mathbf{N}} = F_{N_1} \otimes \cdots \otimes F_{N_d},
\]
where each \( F_{N_i} \) denotes the one-dimensional DFT operator acting on \( \mathbb{C}^{N_i} \).
Here, we use the notation ${\bf N} = (N_1, \cdots, N_d)$. 
 
\vskip.1in

To see how this works in two dimensions, let \( d = 2 \), with \( N_1 = N_2 = 2 \). The 1D  DFT matrix of size \( 2 \times 2 \) is
\[
F_2 = \frac{1}{\sqrt{2}} \begin{bmatrix}
1 & 1 \\
1 & -1
\end{bmatrix}.
\]
Then the 2D DFT on \( \mathbb{C}^{2 \times 2} \) can be expressed as the tensor product
\[
\mathbf{F}_{(2,2)} = F_2 \otimes F_2 =  
\begin{bmatrix}
1 & 1 \\
1 & -1
\end{bmatrix}
\otimes
\begin{bmatrix}
1 & 1 \\
1 & -1
\end{bmatrix}.
\]
Evaluating this tensor product gives a \( 4 \times 4 \) matrix acting on vectorized \( 2 \times 2 \) inputs:
\[
\mathbf{F}_{(2,2)} = \frac{1}{2} 
\begin{bmatrix}
1 & 1 & 1 & 1 \\
1 & -1 & 1 & -1 \\
1 & 1 & -1 & -1 \\
1 & -1 & -1 & 1
\end{bmatrix}.
\]

This matrix performs the 2D DFT by simultaneously transforming along rows and columns. It is important to note that although \( F_2 \otimes F_2 \) and the DFT  \( F_4 \) on $\mathbb C^4$ are both \( 4 \times 4 \) matrices, they represent fundamentally different transforms. The matrix \( F_4 \) corresponds to the one-dimensional DFT on \( \mathbb{C}^4 \), involving 4th roots of unity and acting on 1D signals of length 4. In contrast, \( F_2 \otimes F_2 \) is the tensor (Kronecker) product of two \( F_2 \) matrices and corresponds to a two-dimensional separable DFT on \( \mathbb{C}^{2 \times 2} \), acting on 2D signals. The two transforms differ both in spectral content and  structure. 
 
\medskip 

 {\it Remark.}
 (2D vs.\ 1D DFT).
The 2D (two–dimensional) DFT \(F_{N_1}\otimes F_{N_2}\) is the Fourier transform on the  tensor product 
\(\mathbb{C}^{N}\otimes\mathbb{C}^{N}\), with underlying group  \(\mathbb{Z}_{N_1}\times\mathbb{Z}_{N_2}\), 
whose characters are
\[(n_1,n_2)\mapsto e^{-2\pi j\bigl(k_1 n_1/N_1 + k_2 n_2/N_2\bigr)}.\]
In contrast, the 1D (one-dimensional) DFT \(F_{N_1N_2}\) acts on $\mathbb{C}^{N_1N_2}$ the underlying  cyclic group
\(\mathbb{Z}_{N_1N_2}\),
with characters
\(n\mapsto e^{-2\pi j k n/(N_1N_2)}\).
Although both transforms are represented by \(N_1N_2\times N_1N_2\) matrices, the underlying groups are non-isomorphic whenever \(N_1,N_2>1\); thus they differ in spectral basis and separable structure.
Flattening an \(N_1\times N_2\) array into a length-\(N_1N_2\) vector does \emph{not} convert the 2D DFT into \(F_{N_1N_2}\).

\subsection{Limiting operators}

 \subsubsection{Domain-limiting operators}\label{limiting-operators}
 
Let \( x \in \mathbb{C}^{N_1 \times N_2 \times \cdots \times N_d} \) be a \(d\)-dimensional signal defined on the grid $\cI$ \eqref{grid}.  Let \begin{align}\label{time-concentration-grid}
\mathcal{I}_T =[0,M_1] \times  \cdots \times [0,M_d-1],
\end{align}
be a {\it hyper-rectangular box}, 
with \( M_i \leq N_i \) for \( i=1,\ldots,d \). We say the  signal 
$x$  is \textit{domainlimited} to   $\cI_T$ 
 if it satisfies 
\!
$$
x[\textbf{n}] = 0 \quad \text{for all} \quad \textbf{n}\in \cI\setminus \cI_T.
$$
    
We call $\cI_T$ the \textit{concentration  domain}.

We define {\it the  domain–limiting operator}  \( T_{\mathbf{M}} \) as the operator that restricts the signal \( x \) to a    concentration domain $\cI_T$.  That is, the operator \( T_\textbf{M} : \mathbb{C}^{N_1 \times \cdots \times N_d} \to \mathbb{C}^{N_1 \times \cdots \times N_d} \) is defined element-wise as
\[
(T_{\mathbf{M}}x)[n_1, n_2, \ldots, n_d] =
\begin{cases}
x[n_1, n_2, \ldots, n_d], & \text{if } (n_1, n_2, \ldots, n_d) \in \mathcal{I}_T, \\
0, & \text{otherwise.}
\end{cases}
\]
  
 \medskip 
 
Notice that when $M_i=N_i$ for all $i$,  the operator $T_{\bf M}$ is the identity map.  
\medskip 

To illustrate this  in one-dimensional case,   
let \( x \in \mathbb{C}^N \) be a 1D signal indexed by \( n \in [0,N] \). The domain-limiting operator   \( T_M : \mathbb{C}^N \to \mathbb{C}^N \), with \( M \leq N \),  is known as  {\it time-limiting operator}, and is defined by
\[
(T_M x)[n] =
\begin{cases}
x[n], & \text{if } n \in [0,M], \\
0, & \text{otherwise}.
\end{cases}
\]

Hence, any 1D signal whose support is contained in the interval $[0,M]$ is called  a \textit{time-limited} signal with a time duration (or time-width) of $M$.

  Note that 
 the domain-limiting operator $T_\textbf{M}$ is the tensor product of the  time–limiting operators in dimension 1D.
  Indeed, we have 
 the signal space (see Apendix \ref{appendix:proof-isomorphism} for the proof)
\[
  \mathbb{C}^{N_1\times\cdots\times N_d}
  \;\cong\;
  \mathbb{C}^{N_1}\otimes\cdots\otimes\mathbb{C}^{N_d},
\]
the $d$-dimensional domain-limiting operator splits into a Kronecker
product of its one-dimensional versions:
\begin{align}\label{time-limiting-op-HD}
  T_{\mathbf{M}}
  \;=\;
  T_{M_1}\otimes T_{M_2}\otimes\cdots\otimes T_{M_d}.
\end{align}

For a separable signal $x=x_1\otimes\cdots\otimes x_d$,
\[
  (T_{\mathbf{M}}x)[n_1,\dots,n_d]
  \;=\;
  \bigl(T_{M_1}x_1\bigr)[n_1]\,
  \cdots\,
  \bigl(T_{M_d}x_d\bigr)[n_d],
\]
which matches the element-wise definition of $T_{\mathbf{M}}$ on the tensor product of space $\otimes_{i=1}^d \mathbb{C}^{N_i}$, thus on $\mathbb{C}^{N_1\times \cdots \times N_d}$ by canonical linear isomorphism.  
 This completes our claim.

\medskip


The following lemma is stated  for $T_M$ in 1D. However, the result also holds for $T_{\bf M}$ in higher dimensions using the representation  
\eqref{time-limiting-op-HD}
and  the definition \eqref{composition-of-tensored-op} .


\begin{lemma}\label{timelimitopertorisaprojection} The operator $T_{M}$ is an orthogonal projection map  onto the space of  time-limited signals with duration of $M$. 
\end{lemma}
\begin{proof}
     We need to prove that $T_M$ is self-adjoint and idempotent. By the definition of the map $T_M$, the matrix representation of this map is a diagonal matrix with $1$ in the positions $(m,m)$, $m\in [0,M]$, and $0$ elsewhere. This proves that $T_M$ is symmetry. To prove that $T_M$ is idempotent, 
  assume that $x$ is a time-limited
  signal of duration $M$. That is   $x[n] = 0$ if $n \notin M$. Then, by the definition of $T_M$,  
    $$T_Mx=x$$
    Clearly, $T_M^2=T_M$ because by applying $T_M$ once to any signal, it becomes an  time-limited signal and $T_Mx=x$ for any time-limited signal $x$. 
\end{proof}

The conclusion of the lemma also holds for $T_{\mathbf{M}}$. 
\vskip.1in

Next, we define the frequency-limiting signals  and frequency-limiting operators,  the dual of the time-limiting  signals and time-limiting operators in the frequency domain.

\subsubsection{Band–limiting operators} 

In parallel with the concept of time-limited signals, we first define their dual in the frequency domain.

\begin{definition}
Let \( N, K \in \mathbb{N} \)  
such that  
\( K \leq \lfloor (N-1)/2 \rfloor \). A signal \( x \in \mathbb{C}^N \) is said to be \emph{\(K\)-band-limited} if its discrete Fourier transform \( \widehat{x} \in \mathbb{C}^N \) satisfies
\[
\widehat{x}(\ell) = 0 \quad \text{for all } \ell \in [N] \setminus \{-K, \ldots, K\}.
\]
\end{definition}
We shall denote the space of such signals as $B(K)$. 

\medskip 

Before we define band-limiting operators, we need to introduce a kernel. 
Let \( N \in \mathbb{N} \) be the ambient signal length, and 
$K<N$ such that $2K+1<N$. 
The \emph{Dirichlet kernel} \( D : [0,N] \to \mathbb{C} \) is defined by
\begin{align}\label{DrichletKernel}
D[n] = \sum_{k = -K}^{K} e^{2\pi j k n / N}, \quad n \in  [0,N]. 
\end{align}
  This kernel is the inverse discrete Fourier transform of the indicator function on the frequency band \( \{-K, \ldots, K\} \), and plays a key role in describing bandlimited signals on \( [0,N] \).
If we denote by  $1_K$   the indicator function of the set, then $1_K(k)=1$ when $k=-K, \cdots, K$,  and zero elsewhere. 
\begin{lemma}\label{Drichlet-as-INT}
    For \( n \not\equiv 0 \mod N \), the Dirichlet  kernel  admits the closed-form expression
\begin{align} 
D[n] =  \frac{\sin\!\Bigl(2W\pi n \Bigr)}{\sin\!\Bigl(\pi \frac{n}{N}\Bigr)}, 
\end{align}
with \( D[0] =2WN \), where \( W  \in (0,1/2) \) and  given by $W=\cfrac{2K+1}{2N}$. 
Moreover, $D= F^{-1}(1_K)$. 
\end{lemma}

  \begin{proof}
      Notice that   the sum in the definition of $D$  in \eqref{DrichletKernel}  is a well-known finite geometric series which yields 
      \begin{equation} 
 \sum_{k=-K}^K   \, e^{2\pi j  \frac{k n}{N}}   = \frac{\sin\!\Bigl((2K+1)\pi \frac{n}{N}\Bigr)}{\sin\!\Bigl(\pi \frac{n}{N}\Bigr)}, \quad n\neq 0,  
\end{equation}
with the sum equal to $2K+1=2NW$ when $n=0$.

The proof of $D= F^{-1}(1_K)$ follows from the definition of the inverse Fourier transform. 
 
  \end{proof}

 {\it Notation:} 
 We call $W$ \textit{bandwidth}, hence we write  $D_W$ to emphasize the dependence of the Dirichlet kernel on $W$. 

\medskip

Now we are ready to define bandlimited operators. First, 
we define them in 1D  case. 
 When \( d = 1 \), then   \(\mathcal{I}_B = \{-K, \ldots, K\}\), and  the bandlimiting operator \(B_{K}\) is defined by
\[
B_{K} = F_{N}^{-1}\, T_{K}\, F_{N},
\]
where \(N\) is the dimension of the Fourier space (as defined above),
  \(K< N\) such that $2K+1< N$,
 \(F_{N}\) denotes the Fourier transform on \(\mathbb{C}^{N}\), and
 \(T_{K}\) is the one-dimensional time-limiting operator with  the frequency concentration domain  $\{-K, \cdots, K\}$.

\medskip 

\begin{lemma}\label{convolution}
    The band-limiting operator $B_K$ with the frequency concenteration domain $\{-K, \cdots, K\}$ and the bandwidth $W=\frac{2K+1}{2N}$  is given   
      as a convolution  map by 

$$
B_{K}(x) = x\ast D_W.  
$$

\end{lemma}

\begin{proof}  Let $x$ be a signal defined as follows:
\begin{align}
   x^T &= \begin{bmatrix}
           x(0) & x(1) & \cdots & x(N-1)
         \end{bmatrix}
\end{align}
Observe that for a fixed position $j\in [N]$, the following is true 
\begin{align*}
    [B_Kx](j)&=\frac{1}{N}\sum_{l=0}^{N-1}\sum_{m\in K}\omega^{m(j-l)}x(l)\\
    &=\sum_{l=0}^{N-1}\left[\frac{1}{N}\sum_{m=0}^{N-1}1_K(m)\omega^{m(j-l)}\right]x(l)\\
    &=[x\ast F^{-1}(1_K)](j), 
\end{align*}
where the symbol $\ast$ denotes discrete convolution, and $\omega$ is the $N$-th root of unity.   Therefore, by Lemma \ref{Drichlet-as-INT}, we may express the band-limiting operator as a convolution:
\[
B_K x = x \ast D_W,
 \qquad
D_W = F^{-1}(\mathbf{1}_K),
\]
where \(\mathbf{1}_K\) denotes the indicator function of the frequency band \(K\). This completes the proof.
\end{proof}

\begin{lemma}\label{projection} Let $K$ be as above. The operator $B_{K}$ is a projection map onto   $B(K)$, the space of $K$-band-limited signals with frequency band $W=\frac{2K+1}{2N}$. 
\end{lemma}
\begin{proof}
    Let $x\in B(K)$  and $B_K$ described as above. Then by Lemma \ref{convolution}
   \begin{align*}
B_K^2 x &= B_K(x * D_W) = (x * D_W) * D_W \\[2pt]
        &= F^{-1}\!\bigl(\hat{x}\,\mathbf{1}_K^2\bigr)
         = F^{-1}\!\bigl(\hat{x}\,\mathbf{1}_K\bigr) \\[2pt]
        &= x * D_W = B_K x.
\end{align*}

    The symmetry property $B_K$ is trivial, as   it is a composition of  symmetric maps $F, F^{-1}$ and $T_N$. This completes the proof. 
\end{proof}

The definition of the Dirichlet kernel can be extended to higher dimensions as follows: 
\medskip 

Let $N_i$, $1\leq i\leq d$ be   non-negative integer numbers. Let $K_i<N_i$ such that $2K_i+1< N_i$. 
Define a multi-indexing set $\cI_B$ as follows:  
\begin{align}\label{frequencygrid2}
\mathcal{I}_B = \{ -K_1, \ldots, K_1 \} \times \cdots \times \{ -K_d, \ldots, K_d \}, 
\end{align}
Denote ${\bf W} = (W_1, \cdots, W_d)$, and ${\bf K}= (K_1, \cdots, K_d)$. 
The {\it multi-dimentional Dirichlet kernel} is defined by 
\!
\begin{align}\label{Drichlet-D}
    D_{\bf W}({\bf n}) = \prod_{i=1}^d \frac{\sin(2W_i\pi n_i)}{\sin(\pi \frac{n_i}{N_i})},  \qquad {\bf n}\neq 0, 
\end{align}
\!
 and $D_\textbf{W}(\textbf{0})=2^d \prod_{i=1}^d W_iN_i$. 

 \medskip

Now we are ready to define bandlimited operators.

\begin{definition}[Band-limiting operators]\label{HD-FL} 
For given $N_i$,  $K_i$, and $W_i$ as above, we define  the     band-limiting operator  \( B_{\bf K} : \mathbb{C}^{N_1 \times \cdots \times N_d} \to \mathbb{C}^{N_1 \times \cdots \times N_d} \) as a convolution map 
 by 
  \! 
\begin{align}\label{HD-BLT}
B_{\bf K}(x) = x\ast D_{\bf W}. 
\end{align}
    
\end{definition}

One can easily verify that, due to the separable property of the tensor product of the  one dimensional Fourier transforms and Lemma \ref{Drichlet-as-INT}, the multidimensional Dirichlet kernel satisfies
\[
D_{\mathbf{W}} = \mathbf{F}^{-1}( \mathbf{1}_{\mathbf{K}} ),
\]
where \( \mathbf{1}_{\mathbf{K}} \) denotes the indicator function of the frequency set defined in~\eqref{frequencygrid2}. That is,
\[
\mathbf{1}_{\mathbf{K}}(\mathbf{n}) = 
\begin{cases}
1, & \text{if } \mathbf{n} \in \mathcal{I}_B, \\
0, & \text{otherwise},
\end{cases}
\]
and  \( \mathcal{I}_B \) is  the bandlimiting index set defined in \eqref{frequencygrid2}.

Therefore, the projection operator onto the space of \(\mathbf{K}\)-bandlimited signals can be written as
\begin{align*} 
B_{\mathbf{K}} = \mathbf{F}^{-1} T_{\mathbf{K}} \mathbf{F},
\end{align*}
where \( T_{\mathbf{K}} \) is the frequency-domain-limiting operator (a diagonal matrix) that zeroes out all DFT coefficients outside \( \mathcal{I}_B \) and retains those inside.

\medskip 

The map $B_{\bf K}$ (using the definition of the convolution) is then given by 
\! 
\[
(B_{\bf K}x)[\mathbf{n}] = \frac{1}{N_1 \cdots N_d} \sum_{\mathbf{m}\in \mathcal{I}} x[\mathbf{m}]  
D_{\textbf{W}}(\textbf{n-m})
= 
\frac{1}{N_1 \cdots N_d} \sum_{\mathbf{m}\in \mathcal{I}} x[\mathbf{m}] \prod_{i=1}^d 
D_{W_i}(n_i-m_i)
%
\qquad {\bf n}\in \cI.
\]

Therefore, we have established the following result, which highlights the tensor product structure of the operator $B_{\bf K}$: 
 
\begin{lemma}  Let 
$B_{\bf K}$ be the band-limiting operator defined   in \eqref{HD-BLT}. Then 
\!
  $$(B_{\bf K} x)[{\bf n}] = 
    \prod_{i=1}^d  (B_{K_i})(n_i).$$
\end{lemma}

\noindent {\it Notation:} Similar to the 1D case, we shall denote by $B({\bf K})$ the space of all ${\bf K}$-bandlimited signals in $\mathbb C^{N_1\times \cdots \times N_d}$.

\section{Multi-indexed  prolate matrices and their spectrums}\label{prolatematrix}

Let \( x \in \mathbb{C}^{N_1 \times \cdots \times N_d} \) be a \( d \)-dimensional signal defined on the grid 
 $\mathcal{I}$ \eqref{grid},  and let  \( T_\textbf{M} \) 
be the domain-limiting operator  which  restricts the signal to a smaller region $ 
\mathcal{I}_T$ 
\eqref{time-concentration-grid}. 
Let the band–limiting operator \( B_{\mathbf{K}} \) act by zeroing out the multidimensional Fourier coefficients outside a prescribed frequency support set
$ 
\mathcal{I}_B$ (as defined  in \eqref{frequencygrid2}).

The \emph{(multidimensional) prolate matrix}   is then defined as the composition
\[
A = T_{\mathbf{M}} \, B_{\mathbf{K}} \, T_{\mathbf{M}}.
\]
This operator acts as follows:
  \( T_{\mathbf{M}} \) first truncates the signal \( x \) to the spatial region \(\mathcal{I}_T\). The operator 
   \( B_{\mathbf{K}} \) then bandlimits the truncated signal by retaining only the DFT coefficients with frequencies in \(\mathcal{I}_B\).
  Finally, the second application of \( T_{\mathbf{M}} \) ensures that the resulting signal is confined to the same spatial region \(\mathcal{I}_T\).

The following result proves that the prolate matrix is a Toeplitz matrix, meaning that there exists a kernel function \(K\) such that its entries satisfy
\[
A_{\mathbf{m},\mathbf{n}} = K(\mathbf{m}-\mathbf{n}) =K(\mathbf{n}-\mathbf{m}) .
\]

 \begin{lemma}\label{lem:kernel}  The prolate matrix \(A\) is a Toeplitz matrix with kernel
\[
K: \mathbb{C}^{M_1 \times \cdots \times M_d} \to \mathbb{R},
\]
defined by
\!
\begin{align}\label{kernel}
K(\mathbf{m}) = \prod_{i=1}^d \frac{D_{W_i}(m_i)}{N_i}   \quad \mathbf{m} \neq \mathbf{0},
\end{align}

and
 $ 
K(\mathbf{0}) =  2^d \prod_{i=1}^d W_i$, \; $W_i= \frac{2K_i+1}{2N_i}$.
 
\end{lemma}

 \begin{proof} Multiplying \(B_{\mathbf{K}}\) on the right by \(T_{\mathbf{M}}\) retains only the columns of \(B_{\mathbf{K}}\) corresponding to the indices in \(\mathcal{I}_T\). On the other hand, multiplying \(B_{\mathbf{K}}\) on the left by \(T_{\mathbf{M}}\) retains only the rows corresponding to the indices in \(\mathcal{I}_T\). Consequently, the resulting matrix is zero everywhere except for entries with both row and column indices in \(\mathcal{I}_T\); in those positions, the entries are given by 
 $ 
B_{\bf K} (\mathbf{m},\mathbf{n})
$, the $(\mathbf{m},\mathbf{n})$-entry of the matrix $B_{\bf K}$. Now  define  the square matrix  $B_{\bf M,K}$,   of size $\cI_{T}\times \cI_{T}$,  such that 
 \! 
\begin{align}\label{Kernel2} [B_{\bf M,W}]_{\bf m,n}: = 
 \prod_{i=1}^d \frac{\sin\!\Bigl((2K_i+1)\pi \frac{(n_i-m_i)}{N_i}\Bigr)}{N_i\sin\!\Bigl(\pi \frac{(n_i-m_i)}{N_i}\Bigr)}\qquad \forall \;   \mathbf{m},\mathbf{n} \in \mathcal{I}_T, \, m_i\neq n_i.
\end{align}
(Notice, for $n, m\in [0,M]$, we   have $n-m\in [0,M]$.)  
Note that for any component where $m_i = n_i$, we define the corresponding term in the product to be the time-bandwidth product $2W_iN_i$. (Notice, for $n, m \in [0,M]$, we have $n - m \in [0,M]$.)

It is clear to see that $A_{\bf m,n} =   [B_{\bf M,K}]_{\bf m,n}= K({\bf n}- {\bf m}) =  K({\bf m}- {\bf n})$, where $K$ is given as in
\eqref{kernel}, and 
this completes the  proof. 
 \end{proof}

The kernel in Lemma  \ref{lem:kernel}    is the Drichelet kernel in  the \( d \)-dimensional setting. Moreover, 
for any $n\in \cI_T$
     $$ (Ax)[{\bf n}] =     \sum_{\mathbf{m}\in \mathcal{I}_T } x[\mathbf{m}] K(\mathbf{m}-\mathbf{n})
     $$   
     \! 
 Notice, we have 
$$
A = \otimes_{i=1}^dT_{M_i}B_{K_i} T_{M_i}
$$
The following lemma shows that the  multidimensional prolate matrix  
inherits several important structural and spectral properties from its one-dimensional counterpart. 
\! 
\begin{lemma}\label{spectrum-of-tensors}
     Let $A_i= T_{M_i}B_{K_i}T_{M_i}$, and 
\( 
A \;=\;\otimes_{i=1}^d  A_i.
\)
For each \(i\), denote the spectrum by
\[
\operatorname{Spec}\bigl(A_i\bigr)
=\{\lambda_{i}^{(1)},\dots,\lambda_{i}^{(d_i)}\}.
\]
Then the spectrum of \(A\) is
\[
\operatorname{Spec}(A)
=\Bigl\{\prod_{i=1}^n \lambda_{i}^{(j_i)}
   \;\Bigm|\;
   1\le j_i\le M_i
\Bigr\}.
\]
\end{lemma}

\begin{proof} 
 For each \(i\) choose an eigenbasis 
    \[
A_i\,v_{i,j_i} = \lambda_{i}^{(j_i)}\,v_{i,j_i},
\qquad j_i = 1,\dots,d_i.\]

 For each multi‑index \((j_1,\dots,j_n)\), set
 
\[
v_{j_1,\dots,j_n}
:= v_{1,j_1}\otimes v_{2,j_2}\otimes\cdots\otimes v_{n,j_n}.
\]

Then by the mixed–product property of the tensor product,

\[
\Bigl(\otimes_{i=1}^n A_i\Bigr)\,v_{j_1,\dots,j_n}
= (A_1 v_{1,j_1})\otimes\cdots\otimes(A_n v_{n,j_n})
= (\lambda_{1}^{(j_1)}\cdots\lambda_{n}^{(j_n)})\,v_{j_1\dots,j_n}.
\]
Since there are \(\prod_{i=1}^n d_i\) of these tensors and they are linearly independent, they form an eigenbasis of \(\otimes_{i=1}^n A_i\). Therefore, 
\[
\operatorname{Spec}\Bigl(\otimes_{i=1}^n A_i\Bigr)
=\Bigl\{\lambda_{1}^{(j_1)}\cdots\lambda_{n}^{(j_n)}
  \;\Bigm|\;
  1\le j_i\le M_i\Bigr\},
\]
as claimed. 
\end{proof}

  

\begin{theorem}\label{A-is-selfadjoint}  
 The  multidimensional prolate matrix 
 $ A = T_{\mathbf{M}} B_{\mathbf{K}} T_{\mathbf{M}}
 $ 
is self-adjoint (Hermitian) and positive definite with positive  eigenvalues.  Moreover, all eigenvalues of $A$ are positive and   
lie in the interval \((0, 1)\).   
\end{theorem} 
\begin{proof}
Since both the time-limiting operator \( T_{\mathbf{M}} \) and the band-limiting operator \( B_{\mathbf{K}} \) are orthogonal projections (see Lemmas \ref{timelimitopertorisaprojection} and   \ref{projection}), they are self-adjoint. Therefore, the operator \( A = T_{\mathbf{M}} B_{\mathbf{K}} T_{\mathbf{M}} \), being a composition of self-adjoint operators (with symmetry preserved), is also self-adjoint.

This symmetry implies that 
 all eigenvalues of \(A\) are real. On the other hand, note that for any vector \(x\in \mathbb C^{N_1\times \cdots \times N_d}\),
 \!
\begin{align}\label{A-is-positive}
\langle Ax, x \rangle = \|B_{\mathbf{K}} T_{\mathbf{M}} x\|^2 \ge 0.
\end{align}
Thus, \(A\) is a positive definite operator, which implies that all its eigenvalues are positive.

 Since \(B_{\mathbf{K}}\) is an orthogonal projection operator   and \(T_{\mathbf{N}}\) is a truncation map, the combined effect of these operators ensures that the eigenvalues of \(A\) (with \(A = T_\textbf{M} B_{\mathbf{K}} T_{\mathbf{M}}\)) are bounded above by 1. This is also evident from the following operator norm estimate:
\[
\langle Ax, x\rangle = \|B_{\mathbf{K}} T_{\mathbf{M}} x\|^2 \leq \|x\|^2.
\]

This inequality along with \eqref{A-is-positive}  implies that the eigenvalues of $A$ lie in $[0,1]$.

  The fact that the eigenvalues lie strictly in the interval $(0,1)$ follows from a result in \cite{slepian1978-V}: in the one‑dimensional case, the eigenvalues of the prolate matrix are nonzero and strictly less than 1. Together with Lemma \ref{spectrum-of-tensors}, this completes the proof. 
 \end{proof}

\noindent{\bf Remark:}  (Degenerate eigenvalues of multiplicity $d$.) In \cite{slepian1978-V}, it was shown that, in the one-dimensional case, the eigenvalues of the prolate matrix are non-degenerate. However, in higher dimensions, this property no longer holds in general.

To illustrate this, consider the two-dimensional case ($d = 2$). Let $A$ and $B$ be identical prolate matrices associated with parameters $K_1 = K_2$, $M_1 = M_2$, and $N_1 = N_2$. Then the tensor product matrix $A \otimes B$ defines a higher dimensional prolate matrix. In this case, each eigenvalue of $A \otimes B$ has multiplicity $2$.

This is true  because any eigenvalue $\mu$ of $A \otimes B$ is given by the product of two (multiplicity-1) eigenvalues of $A$, say $\lambda_i$ and $\lambda_j$ for $i \neq j$, i.e., $\mu = \lambda_i \lambda_j = \lambda_j \lambda_i$, (see Lemma \ref{spectrum-of-tensors}). In Fig. \ref{figure:prolate7},  we show the eigenvalues of a prolate matrix 
$A$, along with the eigenvalues of the tensor  product 
$A\otimes A$. The bold dots highlight the multiplicity-2 eigenvalues in the spectrum of the tensored matrix. 
 
This phenomenon generalizes to higher dimensions: for $d>1$, the eigenvalues of the corresponding multidimensional prolate operator  (or matrix) are  degenerate and have   multiplicity  $d$.

\vskip.1in 

\noindent{\bf Remark:}   (Dependence of eigenvalues on the  bandwidth $W$.) 
The prolate eigenvectors $\nu_k$ of the prolate matrix $A=T_MB_KT_M$ and their associated eigenvalues $\lambda_k(M,W)$ can be  considered as the solutions to  the system of equations   
\!
\begin{align}\label{eigenval-equ}
    \lambda_k(M,W) \nu_k[m] = 
    \sum_{n=0}^{M-1}   \frac{\sin\!\Bigl(2W\pi(n-m)\Bigr)}{\sin\!\Bigl(\pi \frac{(n-m)}{N}\Bigr)}  \; \nu_k[n],  
   \qquad \forall    m\in [0,M],  \end{align}

    and when $n=m$, the expression $\frac{\sin\!\Bigl(2W\pi(n-m)\Bigr)}{\sin\!\Bigl(\pi \frac{(n-m)}{N}\Bigr)}$ has the value $2WN$. 
The equation shows that for any   fixed $N$ and $M$ with $M\leq N$, the  eigenvalues $\lambda_k(M,W)$  depend only on  $W$. When $N$,  $M$, and $W$ vary so that time-bandwidth $MW$ remains fixed, Figs.  \ref{figure:prolate4-6} and    \ref{figure:prolate1} show that the eigenvalues remain  approximately unchanged.





\vskip.1in 

\noindent{\bf Remark:}   (Eigenvalue‐graph shape.) 
Slepian,  in \cite{slepian1978-V},  gave the approximation
  \[
    \lambda_{k}(M,W)\;\approx\;\Bigl[\,1 + \exp\!\bigl(
    \tfrac{
    -\pi^2\bigl(2MW - k - \tfrac12\bigr)}{\log(8M\sin(2\pi W)) + \gamma}\bigr)\Bigr]^{-1},
  \]
  valid for \(\lambda_k\in(0.2,0.8)\), where \(\gamma\approx0.5772\) is the Euler–Mascheroni constant.  This suggests that for fixed $M$, $N$ and $W$, the discrete spectrum of the prolate matrix 
  resembles a sigmoid function over the index range corresponding to   eigenvalues  in the range 
$(0.2,0.8)$. See, e.g., Fig. \ref{figure:prolate1}.  

In Fig. \ref{figure:prolate7}, we observe a similar sigmoid-like structure in the discrete spectrum of the tensor product of prolate matrices.   
When $k$, the index of eigenvalue, is fixed, the eigenvalue $\lambda_k$ becomes  a function of $M$ and $W$.   In Fig. \ref{figure:eigenvalues1}, we illustrate function $\lambda_k$ as a function of   the time-bandlimit $2MW$ over  the interval $(0, 400)$.

\subsubsection{Nodal set of Dirichlet kernel}
Although not directly related to the main goal of this paper, the following result is of independent interest and may be relevant in broader contexts involving discrete spectral analysis and time-frequency localization.

\begin{proposition}\label{nodal-Dirchelet}
    Let \( M, N \in \mathbb{N} \) with \( M < N \), and let \( K \in \mathbb{N} \) satisfy \( K \leq \lfloor (N-1)/2 \rfloor \). Define the normalized frequency bandwidth
\[
W := \frac{2K+1}{2N}.
\]
Consider the discrete Dirichlet-type kernel defined for \( m, n \in [0,M] := \{0, 1, \dots, M-1\} \), with \( m \neq n \), by
\[
D_W(n - m) := \frac{\sin\left(2\pi W (n - m)\right)}{\sin\left( \pi \frac{n - m}{N} \right)}.
\]
Define the nodal set of the kernel by
\[
\mathcal{Z}_W := \left\{ (m, n) \in [0,M]^2 \,\middle|\, m \neq n,\; D_W(n - m) = 0 \right\}.
\]
Then the cardinality of the nodal set satisfies the asymptotic estimate

\[
\sharp \mathcal{Z}_W = 2 (M-r) 
\]
for  some integer $r$ satisfying 
\[
\lfloor MW\rfloor -1 \leq r\leq \lfloor MW\rfloor +1.
\] 
That is, the set $\mathcal{Z}_W $ satisfies the following  estimate (up to a constant error): 

\[
|\sharp \mathcal{Z}_W -  2 \lfloor MW\rfloor| \leq  2.
\]

\end{proposition}

\noindent{\bf Remark:}   
This result reflects the fact that the number of oscillations (zero crossings) of the discrete Dirichlet kernel increases proportionally with the time-frequency product \( MW \), paralleling the well-known result that the number of significant  eigenvalues (near $1$) of a time-frequency limiting operator is approximately \( 2MW \). The nodal set structure thus provides insight into the dimensionality and spectral resolution of bandlimited signals in finite settings.

 \vskip.1in 

 \begin{proof}[Proof of Proposition \ref{nodal-Dirchelet}] We provide a sketch of proof here.  
Recall that for \(\Delta = n - m \neq 0\),
\[
D_W(\Delta)
= \frac{\sin\bigl(2\pi W\,\Delta\bigr)}
       {\sin\!\bigl(\pi\,\frac{\Delta}{N}\bigr)}.
\]
The denominator $\sin\!\bigl(\pi\,\frac{\Delta}{N}\bigr)$ never vanishes for \(0<|\Delta|\le M-1<N\), so all zeros come from the numerator. Thus, 
\[
\sin\bigl(2\pi W\,\Delta\bigr)=0
\;\Longleftrightarrow\;
2\pi W\,\Delta = j\pi,
\quad j\in\mathbb Z\setminus\{0\},
\]
hence
\! 
\[
\Delta = \frac{j}{2W}.
\]
We restrict to non-zero integer \(\Delta\) with \(|\Delta|\le M-1\).  Let
\[
J_{\max}
:=\max\bigl\{\ell>0:\;\ell/(2W)\le M-1\bigr\}
=\bigl\lfloor 2W\,(M-1)\bigr\rfloor.
\]
There are exactly \(J_{\max}\) positive integer zeros \(\Delta_1,\dots,\Delta_{J_{\max}}\), and by symmetry each yields two off‑diagonal nodal pairs \((m,n)\) and \((n,m)\).  Hence
\[
\#\mathcal Z_W = 2\,J_{\max}
=2\,\bigl\lfloor2W(M-1)\bigr\rfloor.
\]
Finally, since
\[
2W(M-1) = 2MW - 2W = 2MW + O(1),
\]
rounding down can shift the floor by at most \(1\). Therefore, 
\[
\bigl|\#\mathcal Z_W - 2\lfloor 2 MW\rfloor\bigr|
%
%
\le2,
\]
as claimed.
\! 
\end{proof}

\vskip.1in

 \section{Auxiliary results}\label{auxilliary-results}

We rcall the notations from Section \ref{ssec:main-results}.

   \subsection{Results in one-dimension}
 Recall  the definition of   a one-dimensional prolate matrix,  $A=T_M B_KT_M$, which admits $M$  positive  eigenvalues $\lambda_N^{(k)}(M,W)$ in $[0,1]$. 

\begin{theorem}[\cite{Karnik21}]\label{Karnik21} For any $N, M\in \mathbb N$ with $M\leq N$, and any 
$W, \epsilon\in (0,1/2)$,

    $$\sharp\{k: \lambda_N^{(k)}(M,W)\in (\epsilon, 1-\epsilon)\} \lesssim  R_\epsilon(MW)$$
    where 
    $$ 
   R_\epsilon(MW) =   \frac{2}{\pi^2}\,\log\bigl(100\,M\,W + 25\bigr)\,
\log\!\Bigl(\frac{5}{\epsilon (\,1 - \epsilon\,)}\Bigr)
\;+\;7.
    $$
and  the constant in the inequality   is independent of $N$, $M$ and $\epsilon$ and $W$.
\end{theorem}

We recall the following result from \cite{zhu2017approximating}. 

\begin{theorem}\label{cross-index}
Let $M,N,K\in\mathbb{N}$ satisfy $M<N$ and $2K+1<N$.
Fix a bandwidth parameter $W\in(0,\tfrac12)$. 
Denote by $\lambda_N^{(r)}(M,W)$, $r=0,\dots,M-1$, the eigenvalues of the
prolate matrix 
ordered non-increasingly.

Then
$$
\lambda_N^{(\lfloor 2MW\rfloor -1)}(M,W) \geq 1/2 \geq \lambda_N^{(\lfloor 2MW\rfloor+1)}(M,W), 
$$
\end{theorem}
As a consequence of this lemma, 
 approximately \(2MW\) eigenvalues lie above the \(1/2\) threshold (hence are
close to~\(1\)), while the remaining \(N - 2MW\) eigenvalues lie below
\(1/2\) (and are therefore close to~\(0\)).

\medskip 
By combining Theorems  \ref{Karnik21}  and \ref{cross-index}
 we obtain our next result in 1D setting.
 
 \begin{proposition}\label{1d:prop}
    For any $\gamma\in (0,1)$
    $$  \Bigl|\sharp\{k: ~ 
   \lambda_N^{(k)}(M,W)> \gamma\} -   2MW  \Bigr|\lesssim       R_\gamma(MW)
    $$  
    
    where the constant in the inequality is independent of $M,W$ and $\gamma$.
\end{proposition}

\begin{proof}
   First,  assume that $\gamma\in (0,1/2)$. Then 
  write
\[
\# \{ k : \lambda_N^{(k)}(M,W) > \gamma \} = \# \{ k : \lambda_N^{(k)}(M,W) \in [1/2,1) \} +  \# \{ k : \lambda_N^{(k)}(M,W) \in (\gamma,1/2)\}.
\]
The first term is between $\lfloor 2MW\rfloor -1$ and $\lfloor 2MW\rfloor +1$, by  Theorem \ref{cross-index}, and the fact that the eigenvalues $ \lambda_N^{(k)}(M,W)$ are strictly decreasing (see \cite{slepian1978-V}). 
 The second term is at most $\# \{ k : \lambda_N^{(k)}(M,W) \in (\gamma, 1 - \gamma) \} \lesssim R_\gamma(MW) $,  due to  Theorem   \ref{Karnik21}.  This proves that 
    $$  \Bigl|\sharp\{k: ~ 
   \lambda_N^{(k)}(M,W)> \gamma\} - \lfloor 2MW\rfloor  \Bigr|\lesssim R_\gamma(MW). 
    $$  
  Now, we assume that $\gamma>1/2$.   
Then we can write
\[
\# \{ k : \lambda_N^{(k)}(M,W) > \gamma \} = \# \{ k : \lambda_N^{(k)}(M,W) \in [1/2,1) \} -  \# \{ k : \lambda_N^{(k)}(M,W)\in [1/2, \gamma] \}.
\] 
Again, 
by  Theorem   \ref{cross-index}, the first term is between 
 $\lfloor 2MW\rfloor -1$ and $\lfloor 2MW\rfloor +1$, and the second term is less than $\sharp\{k: \lambda_N^{(k)}(M,W)\in (1-\gamma, \gamma)\}$. 
 Therefore, by applying Theorem   \ref{Karnik21} to  the second part and considering the symmetric  nature of the term $R_\gamma$,   we obtain 

 \[
\# \{ k : \lambda_N^{(k)}(M,W) > \gamma \} \gtrsim
\lfloor 2MW\rfloor -1  - R_\gamma(MW)
\]  
 
 This implies that 
 \[
\Bigl|\# \{ k : \lambda_N^{(k)}(M,W) > \gamma \}  -   \lfloor 2NW\rfloor\Bigr|    \lesssim    R_\gamma(MW). 
\]  

This completes the proof. 
\end{proof}
\subsection{Results in d-dimensions} 

In this section, we establish a series of lemmas in higher dimensions that we will use to prove our main theorem in the next section.

Let $N, M, K \in \mathbb Z$, with $M< N$ and $2K+1< N$. Let $W\in (0,1/2)$. Recall the multi-dimensional prolate matrix, ${\bf A}= T_{\bf M} B_{\bf K} T_{\bf M}$,  introduced in Section      \ref{ssec:main-results},  which admits positive eigenvalues $\lambda_{\bf N}^{(r)}:= \lambda_{\bf N}^{(r)}(\bf M,\bf K)$  lying in the interval $[0,1]$. 
Also, recall the following  notations:

\begin{align*}
&\mathscr{m}_\epsilon(\textbf{M,K}) := \sharp \{r\in \mathbb N: ~ \lambda_{\bf N}^{(r)}>  \epsilon\},  \quad  \epsilon\in (0,1)\\
&\mathscr{n}_\epsilon(\textbf{M,K}) := \sharp \{r\in \mathbb N: ~ \lambda_{\bf N}^{(r)}  \in (\epsilon,1-\epsilon) \}, \quad \epsilon \in (0, 1/2).
\end{align*}


\begin{lemma}\label{lem1} Let $\lambda_{N}^{(r)}:=\lambda_{N}^{(r)}(M,W)$, $1\leq r\leq M$,  denote the eigenvalues of 1D  prolate $A= T_MB_K T_M$. 
Then  for any $\epsilon \in (0,1)$
 \begin{equation}\label{tensor_bd:eqn}
( \# \{ r : \lambda_{N}^{(r)} > \epsilon^{1/d}\} )^d \leq \mathscr{m}_\epsilon ({\bf M}, {\bf K}) \leq (\# \{ r : \lambda_{N}^{(r)} > \epsilon\})^d.
\end{equation}
 
\end{lemma}
\begin{proof} 
   Recall that the   all eigenvalues of  $A$ lie in $(0,1)$. Let $\epsilon\in (0,1)$.  If, for some multi-index \(\ell=(r_1,\dots,r_d)\), $1\leq r_i\leq M$,
\! 
\[
\lambda_{\mathbf N}^{(\ell)}
\;:=\;
\prod_{i=1}^{d}\lambda_N^{(r_i)}
\;>\;\epsilon,
\]
then
\! 
\[
\lambda_N^{(r_i)}>\epsilon
\quad\text{for every } i=1,\dots,d.
\]   
Because each factor \(\lambda_N^{(r_i)}\) lies in \((0,1)\), the product is bounded above by the smallest factor:
\[
\prod_{i=1}^{d}\lambda_N^{(r_i)}
\;\le\;
\min_{1\le i\le d}\lambda_N^{(r_i)}.
\]

Therefore, if the product exceeds \(\epsilon\), the minimum, and hence every factor, must also exceed \(\epsilon\). 
 Conversely, if $\lambda_N^{(r_i)} > \epsilon^{1/d}$ for all $i$, then $\lambda_{\bf N}^{(\ell)}= \prod_{k=1
}^d \lambda_N^{(r_i)} > \epsilon$. 
Therefore, 
\[
\bigl\{\, r : \lambda_{N}^{(r)} > \epsilon^{1/d} \bigr\}^{\, d}
\;\subseteq\;
\Bigl\{\, (r_{1},\ldots,r_{d}) \in \{1,\dots,M\}^{d} :
        \prod_{k=1}^{d} \lambda_{N}^{(r_k)} > \epsilon \Bigr\}
\;\subseteq\;
\bigl\{\, r : \lambda_{N}^{(r)} > \epsilon \bigr\}^{\, d} 
\]

This implies that 
\! 
\[
\Bigl(\sharp \bigl\{\, r : \lambda_{N}^{(r)}  > \epsilon^{1/d} \bigr\}\Bigr)^{\, d}
\;\leq \;
 \mathscr{m}_\epsilon ({\bf M}, {\bf K})^d 
\;\leq\;
 \Bigl(\sharp \bigl\{\, r : \lambda_{N}^{(r)}  > \epsilon  \bigr\}\Bigr)^{\, d} 
\]
\end{proof}

We need the following technical lemma: 

\begin{lemma}\label{chi} For $\epsilon\in (0,1)$, define $\chi(\epsilon)$ by  
\[
\chi(\epsilon)\;:=\;\log\!\Bigl(\frac{1}{\epsilon(1-\epsilon)}\Bigr). 
\]

Then

\begin{itemize}
\item[(i)] $\chi(\epsilon)   \leq 2 \log(1/\epsilon)$ for all $\epsilon\in (0,1/2)$.   
 \item[(ii)]
 \(\chi(\epsilon^{1/d}) \lesssim  \chi(\epsilon)\), where the constant is  
\(C_d \;=\; 2 +\frac{\log(d)}{\log(4)}\), for all $\epsilon\in (0,1)$. 
\end{itemize}
\end{lemma}
\begin{proof} 

(i):   Assume that $\epsilon\leq 1/2$. Then 
$\chi(\epsilon) \leq \log(2/\epsilon)\leq 2 \log(1/\epsilon)$.  

\medskip 

(ii) 
First, write 
\[
\chi\bigl(\epsilon^{1/d}\bigr)
=\log\!\Bigl(\frac{1}{\,\epsilon^{1/d}(1-\epsilon^{1/d})}\Bigr)
=\underbrace{\frac{1}{d}\,\log\!\bigl(\tfrac{1}{\epsilon}\bigr)}_{(I)}
\;+\;
\underbrace{\log\!\Bigl(\frac{1}{\,1-\epsilon^{1/d}\,}\Bigr)}_{(II)}.
\]


Since 
\(\chi(\epsilon)=\log\!\bigl(\tfrac{1}{\epsilon}\bigr)
+\log\!\bigl(\tfrac{1}{1-\epsilon}\bigr)\), 
we have 
\[
(I)
=\frac{1}{d}\,\log\!\bigl(\tfrac{1}{\epsilon}\bigr)
\;\le\;\log\!\bigl(\tfrac{1}{\epsilon}\bigr)
\;\le\;\chi(\epsilon).
\]
Therefore,
\[
(I)\;\le\;\chi(\epsilon).
\]

Now, we find an upper bound for (II). 
Observe that for any \(y\in(0,1)\), 
\[
1 - y^{1/d}
\;=\;
1 - \exp\!\Bigl(\tfrac{1}{d}\,\log(y)\Bigr).
\]
Since the function \(f(t)=1-e^{t}\) is concave on \((-\infty,0]\), 
we have for \(t=\log(y)\le0\):
\[
1 - y^{1/d}
=1 - e^{\,t/d}
\;\ge\;\frac{1}{d}\,\bigl(1 - e^{\,t}\bigr)
=\frac{1}{d}(1 - y).
\]
In particular, with \(y=\epsilon\),
\[
1 - \epsilon^{1/d}
\;\ge\;\frac{\,1-\epsilon\,}{\,d\,}
\quad\Longrightarrow\quad
\frac{1}{\,1 - \epsilon^{1/d}\,}
\;\le\;
\frac{\,d\,}{\,1 - \epsilon\,}.
\]
Taking logarithms gives
\[
(II)
=\log\!\Bigl(\frac{1}{\,1 - \epsilon^{1/d}\,}\Bigr)
\;\le\;
\log\!\Bigl(\frac{d}{\,1 - \epsilon\,}\Bigr)
=\log(d)\;+\;\log\!\Bigl(\tfrac{1}{\,1-\epsilon\,}\Bigr).
\]
Since \(\log\bigl(1/(1-\epsilon)\bigr)\le\chi(\epsilon)\), it follows that
\[
(II)
\;\le\;
\log(d)\;+\;\chi(\epsilon).
\]

\medskip

Combining the bounds for (I) and (II), yields
\[
\chi\bigl(\epsilon^{1/d}\bigr)
=(I)+(II)
\;\le\;\chi(\epsilon)\;+\;\bigl(\log(d)+\chi(\epsilon)\bigr)
=2\,\chi(\epsilon)\;+\;\log(d).
\]
Noting that \(\chi(\epsilon)\ge\chi\bigl(\tfrac12\bigr)=\log(4)\) for all 
\(\epsilon\in(0,1/2)\), we get
\[
2\,\chi(\epsilon)+\log(d)
\;\le\;
\Bigl(2 + \tfrac{\log(d)}{\log(4)}\Bigr)\,\chi(\epsilon)
= C_{d}\,\chi(\epsilon).
\] 
By the symmetry $\chi(\epsilon) = \chi(1-\epsilon)$, the same estimation also holds for $\epsilon\in (1/2,1)$.

Therefore,
\[
\chi\bigl(\epsilon^{1/d}\bigr)
\;\le\;
C_{d}\,\chi(\epsilon), \quad \forall \, \,  0<\epsilon<1
\]
as claimed.
\end{proof}
 
\begin{lemma}\label{lem2} Let $M, W, K$ are as above. For any $\epsilon\in (0,1)$ 
we have 
    \begin{align}\label{M-2MW}
\Bigl|  \mathscr{m}_\epsilon({\bf M}, {\bf K}) - (2MW)^d\Bigr| \lesssim_d  \log(MW)\,\log(1/\epsilon)\,
\max\{\, (\log(MW)\log(1/\epsilon))^{\,d-1},\,(2MW)^{\,d-1}\}. 
\end{align} 
\end{lemma}

 \begin{proof}
By Proposition \ref{1d:prop},  for any $\epsilon\in (0,1)$
\[
\begin{aligned}
 2MW - C  R_\epsilon(MW) &\leq  \# \{ \ell : \lambda_\ell > \epsilon \}  
&\leq  2MW + C  R_\epsilon(MW), 
\end{aligned}
\] 
where $C$ is a constant independent of $\epsilon$.

Raising both sides to the \(d\)-th power and applying the binomial theorem, we obtain

\begin{align}\label{ineq:1}
 (2MW)^d - \tau_\epsilon(MW) &\leq  \left(\# \{ \ell : \lambda_\ell > \epsilon \}\right)^d 
\leq  (2MW)^d +    \tau_\epsilon(MW),  
\end{align}
 
where 
\[\tau_\epsilon(MW):=\sum_{j=0}^{d-1}
      \binom{d}{\,j+1}\,
      (2MW)^{\,d-j-1}\,
      \bigl(C\,R_\epsilon(MW)\bigr)^{\,j+1}\] 
Note $\binom{d}{j+1} \leq 2^d$, we have 
\! 
\[\tau_\epsilon(MW) \lesssim 2^d \sum_{j=0}^{d-1}
      (2MW)^{\,d-j-1}\,
 R_\epsilon(MW)^{\,j+1}.\] 
The constant in the inequality above depends on 
 $d$.  

Take $A:=2MW$ and $B=R_\epsilon(MW)$. Then 

\[\tau_\epsilon(MW) 
\lesssim  2^d A^{d-1} B \sum_{j=0}^{d-1} \Bigl(\frac{B}{A}\Bigr)^j. 
\]

Since \(\sum_{j=0}^{d-1}x^j \le d\,\max\{1,x^{d-1}\}\) for any \(x>0\),
applying this bound yields 
\[\tau_\epsilon(MW) \lesssim 
   2^d A^{d-1} B \sum_{j=0}^{d-1} \Bigl(\frac{B}{A}\Bigr)^j \leq  (d 2^d) B \max\{ B^{d-1} ,  A^{d-1}\}  
\]

 By substituting the expressions for $A$ and $B$, defining
\[
\chi(\epsilon) := \log\!\bigl(\tfrac{1}{\epsilon(1-\epsilon)}\bigr),
\]
and noting that 
\[
R_\epsilon(MW)\;\lesssim\;\log(MW)\,\chi(\epsilon),
\]
we arrive at:
\begin{align}\label{ineq:tau}
\tau_\epsilon(MW)
\;\lesssim\;
(d\,2^d)\,\log(MW)\,\chi(\epsilon)\,
\max\{\, (\log(MW)\chi(\epsilon))^{\,d-1},\,(2MW)^{\,d-1}\}, 
\end{align}

where the constant in the inequality depends only on the dimension $d$.

Using the estimation  in \eqref{ineq:1}, we obtain 
\begin{align}\label{ineq:d}
\Bigl| \left(\# \{ r : \lambda_N^{(r)}> \epsilon \}\right)^d - (2MW)^d\Bigr| \lesssim_d   \log(MW)\,\chi(\epsilon)\,
\max\{\, (\log(MW)\chi(\epsilon))^{\,d-1},\,(2MW)^{\,d-1}\}
 \end{align}
where the constant depends on dimension $d$. 
We apply the previous inequality with $\epsilon$ replaced by $\epsilon^{1/d}$ to obtain

\begin{align}\label{ineq:1/d}
\Bigl| \left(\# \{r : \lambda_N^{(r)} > \epsilon^{1/d} \}\right)^d - (2MW)^d\Bigr| \lesssim_d  \log(MW)\,\chi(\epsilon^{1/d})\,
\max\{\, (\log(MW)\chi(\epsilon^{1/d}))^{\,d-1},\,(2MW)^{\,d-1}\}
\end{align}
 
Assume $\epsilon\in (0, 1/2)$. 
 By Lemma \ref{chi} (ii)  we have 
  $\chi(\epsilon^{1/d}) \lesssim  \chi(\epsilon)$ for  $\epsilon\in (0,1/2)$. Using this   in the   inequality \eqref{ineq:1/d},  and a combination of this with \eqref{ineq:d} and the result of Lemma \ref{lem1}
  yields, 
\! 
\begin{align}\label{ineq:M-2MW}
\Bigl|  \mathscr{m}_\epsilon({\bf M}, {\bf K}) - (2MW)^d\Bigr| \lesssim_d  \log(MW)\,\chi(\epsilon)\,
\max\{\, (\log(MW)\chi(\epsilon))^{\,d-1},\,(2MW)^{\,d-1}\}.  
\end{align} 
Now we use the inequality $\chi(\epsilon) \lesssim \log(1/\epsilon)$ (see Lemma \ref{chi} (i)) to complete the proof for $\epsilon\in (0,1/2)$. 

\medskip 

Now assume that $1/2<\epsilon<1$. By definition, we have 

$$\mathscr{m}_\epsilon({\bf M}, {\bf K})  \leq \mathscr{m}_{1-\epsilon} ({\bf M}, {\bf K})$$  

The desired result then follows from  the symmetry property of $\chi(\xi)$. 
\end{proof}



\section{Proof of main theorem}\label{proof-of-main-theorem} 
The estimation \eqref{M-2MW1} of the main Theorem \ref{main-theorem}   was proven above in Lemma  \ref{lem2}. 
Here, we prove the estimation \eqref{N} of the main theorem. 

Let $\epsilon\in(0,\tfrac12)$, and recall  that
\[
\mathscr{n}_\epsilon(\mathbf M,\mathbf K)
\;=\;
\#\{\,r : \lambda_{\mathbf N}^{(r)} \in(\epsilon,\,1-\epsilon)\}.
\]
Since
\[ 
\{\lambda_{\mathbf N}^{(r)} \in(\epsilon,\,1-\epsilon)\} = 
\{\lambda_{\mathbf N}^{(r)} >\epsilon,\;\lambda_{\mathbf N}^{(r)} <1\}
\;\setminus\;
\{\lambda_{\mathbf N}^{(r)} >1-\epsilon\}, 
\]
we get
\[
\mathscr{n}_\epsilon(\mathbf M,\mathbf K)
\;\le\;
\mathscr{m}_\epsilon(\mathbf M,\mathbf K)
\;-\;
\mathscr{m}_{1-\epsilon}(\mathbf M,\mathbf K).
\]
Therefore
\[
\begin{aligned}
\mathscr{n}_\epsilon(\mathbf M,\mathbf K)
&\;\le\;
\Bigl|\,\mathscr{m}_\epsilon(\mathbf M,\mathbf K) - (2MW)^d\Bigr|
\;+\;
\Bigl|\,\mathscr{m}_{1-\epsilon}(\mathbf M,\mathbf K) - (2MW)^d\Bigr|.
\end{aligned}
\]
Applying Lemma \ref{lem2} to both terms on the right-hand side, and using the fact that
\(\log\!\bigl(1/(1-\epsilon)\bigr) \lesssim \log\!\bigl(1/\epsilon\bigr)\) for \(\epsilon\in(0,\tfrac12)\), we obtain
\[
\mathscr{n}_\epsilon(\mathbf M,\mathbf K)
\;\lesssim_d\;
\log(MW)\,\log\!\bigl(\tfrac1\epsilon\bigr)\,
\max\Bigl\{\bigl(\log(MW)\,\log\tfrac1\epsilon\bigr)^{\,d-1},\;(2MW)^{\,d-1}\Bigr\}.
\]
Hence Lemma \ref{lem2} finishes the proof of the main theorem.

\medskip

\printbibliography

 \clearpage

\appendix

\section{Figures}\label{appendix:figures}

\begin{figure}[ht]
\centering

\includegraphics[width=1\linewidth]{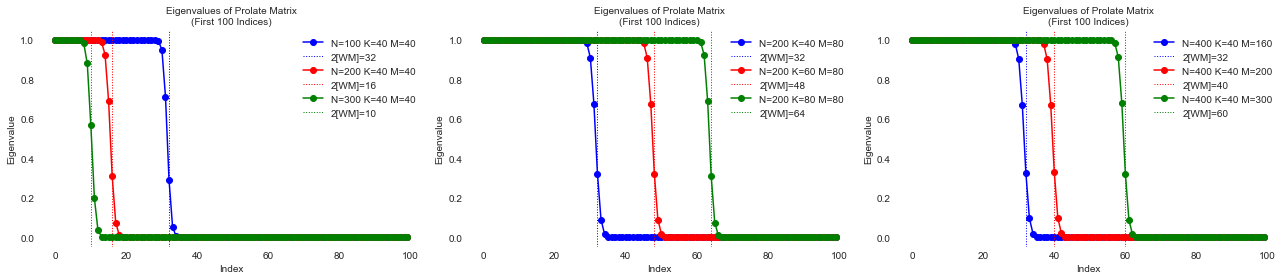}
\caption{\small Influence of time-bandwidth bound- Eigenvalue spectra of time--frequency limiting (prolate) matrices for increasing ambient dimension \( N \), while keeping the time-bandwidth product \( 2MW \) fixed. Each subplot shows the first 100 eigenvalues of the prolate matrix \( A = T_M B_K T_M \), where \( B_K \) is the frequency-limiting operator to bandwidth \( K \), and \( T_M \) is the time-limiting operator to \( M \) samples. Despite increasing \( N \), the number of significant ($\epsilon$-close to 1) eigenvalues remains approximately constant, consistent with the theoretical time-bandwidth bound \( 2MW \). This illustrates that the effective dimensionality of the bandlimited and time-concentrated subspace is controlled primarily by the product \( 2MW \), not the ambient dimension \( N \).
}
\label{figure:prolate9}
\end{figure}

\begin{figure}[ht]
\centering

\includegraphics[width=1\linewidth]{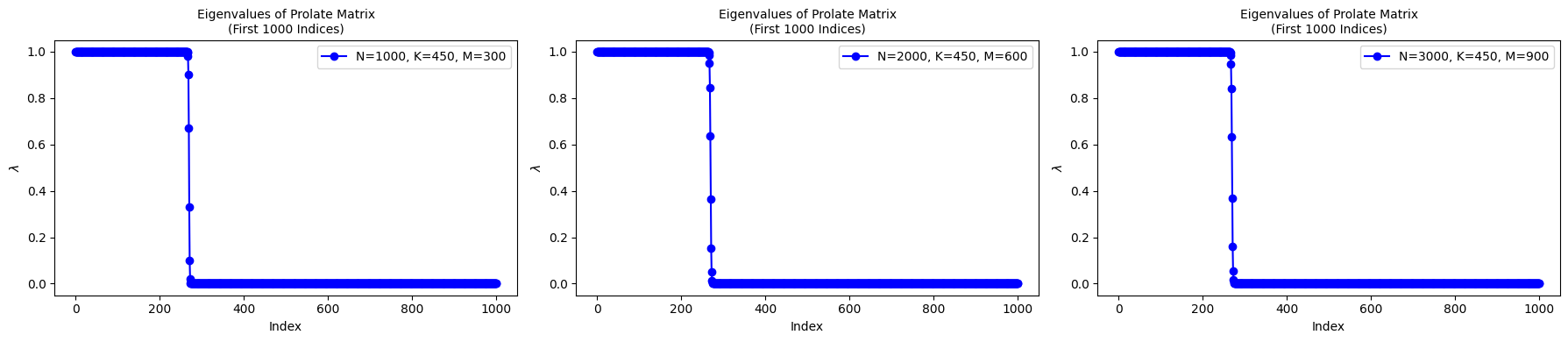}\caption{\small In all these graphs, the time-bandwidth is $2MW \approx 270$. From left to right, the plots correspond to $W \approx 0.45$, $W \approx 0.22$, and $W \approx 0.15$.  
This  shows that the eigenvalue distribution depends primarily on the product $MW$, and remains consistent across different $N, M$ and $W$ configurations as long as   $MW$ held constant.}\label{figure:prolate4-6}
\end{figure}

\begin{figure}[ht]
\centering

\includegraphics[width=1\linewidth]{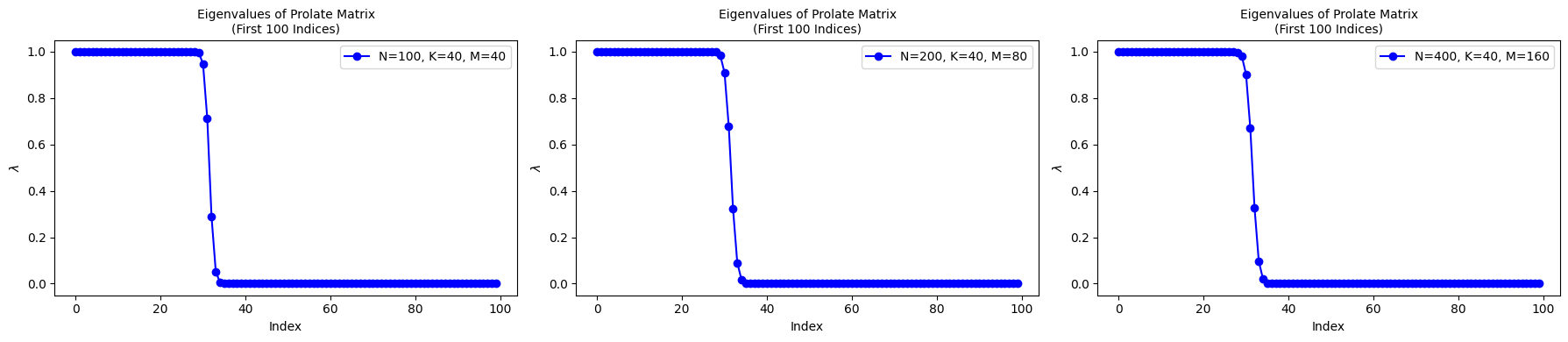}
\caption{Eigenvalue distributions of prolate matrices \( A = T_M B_K T_M \) for increasing ambient dimension \( N \), and parameters $K$ and $M$,  while keeping the time-bandwidth product $2MW$ fixed. Each subplot shows the first $100$ eigenvalues. As \( N \) increases (from left to right), the shape and decay of the eigenvalues remain essentially unchanged, indicating that the spectrum of the time–frequency limiting operator depends primarily on the time–bandwidth product \( 2MW \), and not on the total timewidth \( M \).
 }
\label{figure:prolate1}
\end{figure}


   \begin{figure}[ht]
\centering

\includegraphics[width=0.4\linewidth]{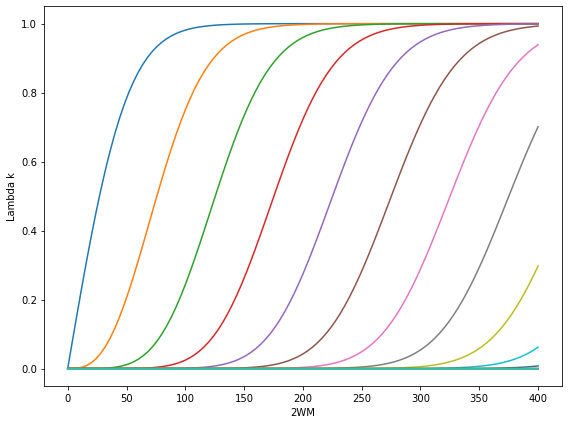}
\caption{\small This graph  shows the plot of eignvalues $\lambda_N^{(\ell)}$ as a function of $MW$ or $W$ for a fixed $M$. Here, we choose $N=1000$, $M=800$, and   allow  $W$ and $2MW$ to vary over the intervals $(0, 0.2)$ and $(0,400)$, respectively. The graphs from left to write show $\lambda_{1000}^{(0)}$, $\lambda_{1000}^{(1)}, \lambda_{1000}^{(2)}, \cdots , \lambda_{1000}^{(k)}, \cdots$. 
}
\label{figure:eigenvalues1}
\end{figure}

 \begin{figure}[ht]
\centering

\includegraphics[width=0.5\linewidth]{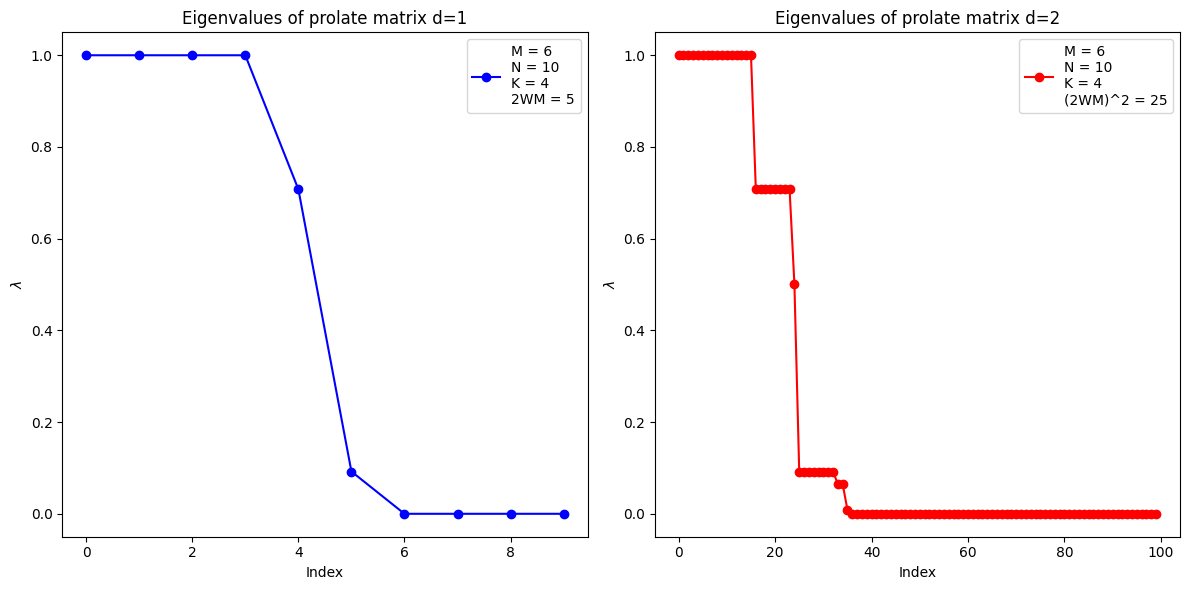}
\caption{\small  Eigenvalue distribution of a one-dimensional prolate matrix and its tensor product. The left plot shows the eigenvalues of a 1D prolate matrix \( A = T_M B_K T_M \), with time-bandwidth product of $2WM=5$. As shown, the leading   eigenvalues are close to $1$, while the remaining ones decay rapidly toward zero. The right plot shows the eigenvalues of the 2D tensor product operator \( A \otimes A \), corresponding to the Kronecker product of the 1D prolate matrix with itself,  with a time-bandwidth product of $25$.  The spectrum consists of all pairwise products of 1D eigenvalues. The multiplicity of each distinct eigenvalue is reflected by the vertical alignment of overlapping points (bold dots).
}
\label{figure:prolate7}
\end{figure}

  \clearpage  

\section{}\label{appendix:proof-isomorphism}

We prove that 
   \[
\mathbb{C}^{N_1 \times \cdots \times N_d} \cong \mathbb{C}^{N_1} \otimes \mathbb{C}^{N_2} \otimes \cdots \otimes \mathbb{C}^{N_d}.  
\]
 We provide the proof for $d=2$, as the argument extends naturally to any $d$.

 Consider the vector space \(\mathbb{C}^{N_1 \times N_2}\) of \(N_1 \times M_2\) matrices over \(\mathbb{C}\). This space has dimension \(N_1N_2\).

Let \(\{e_1, e_2, \ldots, e_{N_1}\}\) be the standard unit basis for \(\mathbb{C}^{N_1}\) and \(\{f_1, f_2, \ldots, f_{N_2}\}\) be the standard unit basis for \(\mathbb{C}^{N_2}\). The tensor product \(\mathbb{C}^{N_1} \otimes \mathbb{C}^{N_2}\) is defined as the vector space spanned by the elementary tensors
\[
\{ e_i \otimes f_j : 1 \leq i \leq N_1, \; 1 \leq j \leq N_2 \}.
\]
Since there are \(N_1\) choices for \(i\) and \(N_2\) choices for \(j\), the set \(\{ e_i \otimes f_j \}\) forms a basis for \(\mathbb{C}^{N_1} \otimes \mathbb{C}^{N_2}\) and the dimension of this tensor product is also \(N_1N_2\).

We can define an isomorphism between \(\mathbb{C}^{N_1 \times N_2}\) and \(\mathbb{C}^{N_1} \otimes \mathbb{C}^{N_2}\) by mapping the elementary tensor \(e_i \otimes f_j\) to the matrix unit \(E_{ij} \in \mathbb{C}^{N_1 \times N_2}\), where
\[
(E_{ij})_{kl} = \delta_{ik}\delta_{jl}.
\]
This map is linear and bijective, and it preserves the vector space structure. Thus, we have the identification:
\[
\mathbb{C}^{N_2 \times N_2} \cong \mathbb{C}^{N_1} \otimes \mathbb{C}^{N_2}.
\]

 \end{document}